\documentclass[a4paper,10pt]{article}
\usepackage[T1]{fontenc}
\usepackage[utf8]{inputenc}
\usepackage{lmodern}
\usepackage[margin=2.42cm]{geometry}
\usepackage{microtype}
\usepackage{amsmath,amsthm,amssymb,mathtools}
\usepackage{enumitem,bm}
\usepackage{braket, booktabs,longtable,caption,subcaption,multirow, makecell,pdflscape}
\usepackage{xparse}
\usepackage{multirow}
\usepackage{color,colortbl}
\usepackage{tikz}
\usepackage{comment}
\usetikzlibrary{cd,calc}
\usepackage{hyperref}
\newcolumntype{B}[1]{>{\begin{minipage}[b]{#1}\raggedright{}}c<{\end{minipage}\minrowheight}}
\newcommand{\minrowheight}{\rule{0pt}{8ex}\relax}

\usepackage{hyperref}
\usepackage{url}

\usepackage{graphicx}

\theoremstyle{plain}
\newtheorem{Thm}{Theorem}
\newtheorem*{Thm*}{Theorem}

\newtheorem{Lem}[Thm]{Lemma}
\newtheorem{Cor}[Thm]{Corollary}
\newtheorem*{Cor*}{Corollary}
\newtheorem{Conj}[Thm]{Conjecture}
\theoremstyle{definition}
\newtheorem{Ex}[Thm]{Example}
\newtheorem{Def}[Thm]{Definition}

\theoremstyle{remark}
\newtheorem{Rem}[Thm]{Remark}
\makeatletter
\def\@tocline#1#2#3#4#5#6#7{\relax
	\ifnum #1>\c@tocdepth 
	\else
	\par \addpenalty\@secpenalty\addvspace{#2}%
	\begingroup \hyphenpenalty\@M
	\@ifempty{#4}{%
		\@tempdima\csname r@tocindent\number#1\endcsname\relax
	}{%
		\@tempdima#4\relax
	}%
	\parindent\z@ \leftskip#3\relax \advance\leftskip\@tempdima\relax
	\rightskip\@pnumwidth plus4em \parfillskip-\@pnumwidth
	#5\leavevmode\hskip-\@tempdima
	\ifcase #1
	\or\or \hskip 1em \or \hskip 2em \else \hskip 3em \fi%
	#6\nobreak\relax
	\hfill\hbox to\@pnumwidth{\@tocpagenum{#7}}\par
	\nobreak
	\endgroup
	\fi}
\makeatother

\numberwithin{Thm}{section}
\allowdisplaybreaks

\newcommand{\Nef}{\operatorname{Nef}}
\newcommand{\Eff}{\operatorname{Eff}}
\newcommand{\Mov}{\operatorname{Mov}}
\newcommand{\Proj}{\operatorname{Proj}}
\newcommand{\Cl}{\operatorname{Cl}}
\newcommand{\Pic}{\operatorname{Pic}}
\newcommand{\Cox}{\operatorname{Cox}}

\newcommand{\Pf}{\operatorname{Pf}}
\newcommand{\GL}{\operatorname{GL}}

\newcommand{\Int}{\operatorname{Int}}

\newcommand{\wt}{\operatorname{wt}}

\newcommand\lcm{\mo{lcm}}

\newcommand\mo[1]{\operatorname{#1}}

\newcommand\actL[1]{\multirow{2}{*}{$#1:$} \close}
\newcommand\actR[1]{\close \multirow{2}{*}{\bigg)#1}}
\newcommand\lBr{\rule{0pt}{2.5ex} \multirow{2}{*}{\bigg(} \close}

\newcommand\close{\!\!\!\!}

\newcommand{\Z}{\mathbb{Z}}
\usepackage[all]{xy}

\author{Livia Campo, Tiago Duarte Guerreiro}
\title{Non-solidity of uniruled varieties}
\date{}

\AtEndDocument{\bigskip{\footnotesize%
		\textsc{School of Mathematics, KIAS, 85 Hoegiro, Dongdaemun-gu, Seoul 02455, Republic of Korea} \par  
		\textit{E-mail address}, L.~Campo: \texttt{liviacampo@kias.re.kr} \par
		\addvspace{\medskipamount}
		\textsc{Department of Mathematical Sciences, University of Essex, Wivenhoe Park, Colchester, CO4 3SQ, UK} \par
		\textit{E-mail address}, T.~Duarte~Guerreiro: \texttt{T.duarteguerreiro@essex.ac.uk} 
}}

\begin{document}
	\makeatletter
	\newcommand{\subjclass}[2][2020]{%
		\let\@oldtitle\@title%
		\gdef\@title{\@oldtitle\footnotetext{#1 \emph{Mathematics subject classification.} #2}}%
	}
	\newcommand{\keywords}[1]{%
		\let\@@oldtitle\@title%
		\gdef\@title{\@@oldtitle\footnotetext{\emph{Key words and phrases.} #1.}}%
	}
	\makeatother

\maketitle

\begin{abstract} 
	We give conditions for a uniruled variety of dimension at least 2 to be non-solid.\ This study provides further evidence to a conjecture by Abban and Okada on the solidity of Fano 3-folds.\ To complement our results we write explicit birational links from Fano 3-folds of high codimension embedded in weighted projective spaces.
\end{abstract}



\section{Introduction}  \label{sect:intro}

We work over the field of complex numbers.\ Let $W$ be a smooth projective variety where $K_W$ is not pseudo-effective. Then, the Minimal Model Program yields a birational model of $V$ of $W$ with a fibre structure $\sigma \colon V \rightarrow S$ of relative Picard rank $1$ where $V$ has mild singularities and $-K_V$ is relatively ample, called a Mori fibre space (see \cite[Corollary~1.3.3]{BCHM}).\ 
The birational classification of Mori fibre spaces can be divided into two classes: those that are birational to a strict fibration, and those that are not.\ The notion that encodes this nature is the one of solidity (cf \cite[Definition~1.4]{hamidokada}).\ We extend their definition to the following.\

\begin{Def}
We say that a uniruled variety of dimension at least two is \textit{birationally solid} if it is not birational to a strict Mori fibre space, that is, to a Mori fibre space $\sigma \colon V \rightarrow S$ where $\dim S >0$.
\end{Def}

Birational solidity implies irrationality in a strong sense.\  In this paper, we establish sufficient conditions for a uniruled variety $X$ to admit a birational map to a strict Mori fibre space.\ Indeed, for $X$ a normal projective variety such that $-K_X$ is $\mathbb{Q}$-Cartier and big, $X$ is uniruled (see Lemma \ref{lem:uni} below).\

\begin{Thm*}[= Theorem \ref{thm:main}]
Let $X$ be a normal Cohen-Macaulay projective variety and $K_X$ be $\mathbb{Q}$-Cartier.\  
Suppose that $-(K_X+lA)$ is big, where $l \coloneqq \lcm{(m_0,m_1)}$ for some $m_0,m_1$ positive integers, and $A \in \Cl(X)$ is an ample $\mathbb{Q}$-Cartier Weil divisor.\ 
If there are two sections $s_i \in H^0(X,m_iA)$ for $i=0,1$ which are independent in the graded ring $R(X,A)\coloneqq \bigoplus_{m\geq 0} H^0(X,mA)$, then $X$ is non-solid.\ 
\end{Thm*}
Let $X$ be Fano $d$-fold with terminal $\mathbb{Q}$-factorial singularities and $A$ a generator of $\Cl(X)$. Then $-K_X = \iota_X A$ and $\iota_X$ is called the Fano index of $X$. We consider $X$ to be polarised by $A$, that is, together with an embedding given by the ring of sections $R(X,A)$.
For each $d$, this produces a list of thousands of candidate Fano $d$-folds embedded in weighted projective spaces. As a result of our main theorem, we prove that Fano varieties tend to stabilise into strict Mori fibrations as its Fano index increases.
\begin{Cor*}[= Corollary \ref{cor:mainold}] Let $X$ be a $\mathbb{Q}$-factorial terminal Fano $d$-fold, with $d\geq 3$, and $A \in \Cl(X)$ a generator of the class group of $X$.\ Consider the embedding given by the ring of sections of $A$
\begin{equation*}
    X \hookrightarrow \mathbb{P}(a_0,\ldots , a_N)
\end{equation*}
Suppose that $\lcm(a_i,a_j) < \iota_X$ for some $i,j \leq N$.\ Then, $X$ is non-solid.\ 
\end{Cor*}
The list of Fano $d$-folds has only been produced for up to $d=3$ and already in this case there are dozens of thousands of candidate Fano 3-folds. Although Corollary \ref{cor:mainold} can be widely applied, it does not detect all Fano 3-folds which are non-solid, (for instance, because the Fano index is too small) and an explicit analysis is needed to complete the picture.\

Concerning those families with Fano index 1 and at most terminal cyclic quotient singularities, there are many birational rigidity results (see \cite{BRhyp,okadaI} for complete intersections,  \cite{hamidokada} for Pfaffian Fano 3-folds, and \cite{okadacod4} for codimension 4 Fano 3-folds).\ 
The situation changes completely for higher Fano index where it has been shown that each is birationally non-rigid and most of them, if not all, are non-solid (cf \cite{hamidvanyapark,GuerreiroThesis}).\ In higher codimension the situation seems to stabilise in the sense that every such Fano is thought to be birational to a strict Mori fibration.\

The explicit construction of Fano 3-folds in codimension 4 and Fano index 2 has been achieved by \cite{CoughlanDucat} first, and then by \cite{CampoC4I2}; in this paper we refer to the latter (see Section \ref{sect:cons} for details).\ We get the following result: 

\begin{Thm} \label{thm:mainC4I2version2}
Let $X$ be a quasi-smooth $\mathbb{Q}$-factorial terminal Fano 3-fold with $-K_X \sim 2A$, where $A$ is a generator of $\Cl(X)$. Suppose that the embedding given by the ring of sections of $A$ is in codimension 4.  Then $X$ is birationally non-rigid. Moreover, if $X$ is not in the families \#39890, \#39928, and \#39660 then there is at least one deformation family of $X$ that is non-solid.\
\end{Thm}

In Table \ref{tab:dimlinsys} we list the families of codimension 4 and Fano index 2 Fano 3-folds that fall in the description of Corollary \ref{cor:mainold}, identified by their Graded Ring Database ID (GRDB, \cite{grdb}).\ We explicitly examine the remaining families in Section \ref{sec:links and mds}, thus proving \ref{thm:mainC4I2version2}.

Our case study also highlights an interesting phenomenon (end of Case II.a below), already occurring in \cite[Section 4.4]{GuerreiroThesis}, that is: in one instance (\#39660) the birational link terminates with a divisorial contraction to a 3-fold embedded in a fake weighted projective space (\cite{buczynska2008fake,KasprzykFakeWPS}).\ Moreover, we are able to compute the Picard rank of the two families \#40671 and \#40672 that behave \`{a} la Ducat \cite{ducat} (Section 4.3), which is equal to 2 in both cases.\ Indeed, this method can be applied in the computation of the Picard rank of certain Fano varieties.

Our work gives evidence to following conjecture

\begin{Conj} \label{conj:1}
Suppose $X$ is a Fano $d$-fold. If $\iota_X$ is high enough then $X$ is non-solid.
\end{Conj}

\paragraph{Acknowledgements}
Both authors would like to express their gratitute to Hamid Abban, Gavin Brown, Ivan Cheltsov and Takuzo Okada.\ The first author is partially supported by the EPSRC grants EP/N022513/1, EP/P021913/1, by the Japan Society for the Promotion of Science (JSPS), and by the Korea Institute for Advanced Study (KIAS), grant No. MG087901.\ The second author is supported by EPSRC grants ref.\ EP/T015896/1 and EP/V055399/1.

\section{Non-solidity}  \label{sect:mainthm}

In this section we prove our main Theorem \ref{thm:main} and Corollary \ref{cor:mainold}.\ The proof of Theorem \ref{thm:mainC4I2version2} is instead contained in Subsection \ref{sub:wrap}.\ 
The authors learnt afterwards about the existence of a more general version of Lemma \ref{lem:uni} (cf \cite[Lemma 3.18]{LazicPeternellTsakanikas}); the following version is nontheless kept for reference. 

\begin{Lem} \label{lem:uni}
Let $X$ be a normal projective variety and $D$ an effective $\mathbb{Q}$-divisor on $X$ such that $-(K_X+D)$ is $\mathbb{Q}$-Cartier and big.\ Then $X$ is uniruled.\
\end{Lem} 

\begin{proof}
We take a resolution of singularities
 $\psi \colon X' \rightarrow X$ and write
\begin{equation*}
    K_{X'}+D'\sim_{\mathbb{Q}} \psi^*(K_{X}+D)+E
\end{equation*}
where $E$ is $\psi$-exceptional and $\psi_*D'=D$. Suppose that $X$ is not uniruled. Then $X'$ is not uniruled and by \cite[Corollary~0.3]{BDPP}, $K_{X'}$ is pseudo-effective.\ Therefore $K_{X} \sim_{\mathbb{Q}} \psi_*K_{X'}$ is pseudo-effective which contradicts the fact that $-K_{X}$ is big.\ We conclude that $X$ is uniruled.
\end{proof}

\begin{Thm} \label{thm:main}
Let $X$ be a normal Cohen-Macaulay projective variety and $K_X$ be $\mathbb{Q}$-Cartier.\  
Suppose that $-(K_X+lA)$ is big, where $l \coloneqq \lcm{(m_0,m_1)}$ for some $m_0,m_1$ positive integers, and $A \in \Cl(X)$ is an ample $\mathbb{Q}$-Cartier Weil divisor.\ 
If there are two sections $s_i \in H^0(X,m_iA)$ for $i=0,1$ which are independent in the graded ring $R(X,A)\coloneqq \bigoplus_{m\geq 0} H^0(X,mA)$, then $X$ is non-solid.\ 
\end{Thm}

\begin{proof}
Since $s_0$ and $s_1$ are independent in $R(X,A)$, the linear system $|lA|$ contains a pencil. Let $\pi \colon X \dashrightarrow \mathbb{P}^1$ be the map given by the sections $s_0$ and $s_1$ and call $F$ the generic fibre of $\pi$.\

Let $\nu \colon F^{\nu} \rightarrow F$ be the normalisation of $F$. By subadjunction, see Lemma \cite[Corollary~5.1.9]{mmprob}, there is an effective $\mathbb{Q}$-divisor $D$ such that
\begin{equation*}
K_{F^{\nu}}+D \sim_{\mathbb{Q}} \nu^*((K_X+F)|_F) \sim_{\mathbb{Q}}
\nu^*((K_X+lA)|_F) \; .
\end{equation*} 
Since $-(K_X+lA)$ is a big $\mathbb{Q}$-Cartier divisor it follows that $-(K_{F^{\nu}}+D)$ is big and $\mathbb{Q}$-Cartier.\ By Lemma \ref{lem:uni}, it follows that $F^{\nu}$ is uniruled.\ Resolving the indeterminacy of $\pi$, we get a commutative diagram,
\begin{center} 
\begin{tikzcd}
 & \widetilde{X} \arrow[swap]{dl}{\psi} \arrow{dr}{\varphi} \\
X \arrow[dashed,swap]{rr}{\pi} && \mathbb{P}^1
\end{tikzcd}
\end{center} 
where $\widetilde{X}$ is a smooth projective variety and $\varphi$ is onto.\ Let $\widetilde{F}$ be the proper transform of $F^\nu$ on $\widetilde{X}$.\ Then, $\widetilde{F}$ is uniruled and we can run a $K_{\widetilde{X}}$-MMP on $\widetilde{X}$ over $\mathbb{P}^1$.\ This is a (finite) sequence of divisorial contractions and small modifications $\chi$, fitting into the diagram
\begin{center}
\begin{tikzcd}
  \widetilde{X} \arrow[dashed]{drr}{\chi} \arrow[swap]{dd}{\varphi} & &  \\
 &  &Y \arrow{d}{\varphi'}  \\
\mathbb{P}^1 && B \arrow{ll}{\sigma}
\end{tikzcd}
\end{center}
where $\varphi' \colon Y \rightarrow B$ is a Mori fibre space where $\dim B > 0$ and $\sigma$ is a surjective morphism with connected fibres.\
\end{proof}

\begin{Cor} \label{cor:mainold} Let $X$ be a $\mathbb{Q}$-factorial terminal Fano $d$-fold, with $d\geq 3$, and $A \in \Cl(X)$ a generator of the class group of $X$.\ Consider the embedding given by the ring of sections of $A$
\begin{equation*}
    X \hookrightarrow \mathbb{P}(a_0,\ldots , a_N)
\end{equation*}
Suppose that $\lcm(a_i,a_j) < \iota_X$ for some $i,j \leq N$.\ Then, $X$ is non-solid.\ 
\end{Cor}
\begin{proof}
Consider the map $\pi \colon X \dashrightarrow \mathbb{P}^1(a_i,a_j)$ which is the restriction of the projection 
\begin{align*}
    \mathbb{P}(a_0,\ldots,a_N)&\dashrightarrow \mathbb{P}^1(a_i,a_j).\\
    [x_0:\ldots:x_N] &\mapsto [x_i:x_j]
\end{align*}
The generic fibre $F$ is a hypersurface in $X$ given by
\begin{equation*}
F \colon (x_i^{l/a_i}+\lambda x_j^ {l/a_j}=0)\subset X
\end{equation*}
and therefore $F\sim lH$. Then $K_X+F$ is $\mathbb{Q}$-Cartier and
\begin{equation*}
   -(K_X+F)\sim (\iota_X-l)H 
\end{equation*}
is ample. By Theorem \ref{thm:main} it follows that $X$ is non-solid.\
\end{proof}

Corollary \ref{cor:mainold} generalises, and retrieves, the result of Abban-Cheltsov-Park \cite[Theorem 1.2]{hamidvanyapark} for Fano 3-fold hyperurfaces.

\section{Case Study: Fano 3-folds in Codimension 4 and index 2}  \label{sect:cons}

In this section we want to discuss a specialisation of Corollary \ref{cor:mainold} when $N=7$ and $\iota_X = 2$ when the dimension of $X$ is 3.\ Corollary \ref{cor:mainold} does not give an explicit description of the birational map between $X$ and the strict Mori fibre space $Y \rightarrow S$.\ 
In this context, we explicitly recover such birational map using the Sarkisov Program.\ 
As an immediate consequence, we have the following corollary.\ 

\begin{Cor} \label{cor:nonsolid 2}
Let $X$ be a family in Table \ref{tab:dimlinsys}.\ Then $X$ is not solid.\
\end{Cor}
\begin{table}[ht]
    \centering
    \begin{tabular}[t]{c c}
        \toprule
        ID                      & $h^0(X,A)$      \\ \midrule
        \#40360   & 2  \\
        \#40370  & 2 \\
        \#40371  & 2  \\
        \#40399   & 2 \\
        \bottomrule
    \end{tabular}\hfill%
    \begin{tabular}[t]{c c}
        \toprule
        ID                    & $h^0(X,A)$      \\ \midrule
        \#40400   & 2  \\
        \#40407   & 2 \\
        \#40663  & 3 \\
        \#40671  & 3  \\
        \bottomrule
    \end{tabular}\hfill%
    \begin{tabular}[t]{c c}
        \toprule
        ID                     & $h^0(X,A)$ \\ \midrule
        \#40672   & 3  \\
        \#40933   & 5 \\
        \#41028  & 8 \\
        \bottomrule
    \end{tabular}
\caption{Families for which $h^0(X,A) \geq 2$}
\label{tab:dimlinsys}
\end{table}
However, there are further 23 Fano 3-folds of index 2 and codimension 4 in the GRDB that are not included in Table \ref{tab:dimlinsys}.\ In Theorem \ref{thm:mainC4I2version2} we claim that, except for three of such Fano 3-folds, at least one deformation family for each of the remaining 20 is non-solid.

We prove Theorem \ref{thm:mainC4I2version2} by studying the birational links initiated by blowing up a Type I centre (or a Type II$_2$ centre for families \#39569 and \#39607) in each of the 23 families, as in Subsection \ref{subsect: Kawamata blowup} and Section \ref{sec:links and mds}.\ 
Such analysis of these birational links relies on the explicit description of index 2 Fano 3-folds in codimension 4.\ 

Several approaches to their explicit construction can be found in the literature.\ Prokhorov and Reid \cite{prokhorovreid} build one family in codimension 4 via a divisorial extraction of a specific curve in $\mathbb{P}^4$ and running the Sarkisov Program starting with such extraction.\ 
In \cite{ducat}, Ducat generalises their construction, recovering the family of Prokhorov-Reid and finding two new deformation families in codimension 4 having the same Hilbert series (cf \cite[Section 3]{T&Jpart1}.\ Of the latter two families, we study only one in this paper (the other can be treated in a similar manner, although we do not include it here for brevity reasons); in particular, we retrieve the birational links of \cite{prokhorovreid,ducat} (see Subsection \ref{subsect: comparison with Ducat}).\ 

To Coughlan and Ducat \cite{CoughlanDucat} it is due a different approach to constructing Fano 3-folds that relies on rank 2 cluster algebras.\

More recently, Campo \cite{CampoC4I2} has constructed a total of 52 families of codimension 4 Fano 3-folds of index 2 by means of equivariant unprojections, some of which correspond to the same Hilbert series, also in accordance to \cite{ducat,CoughlanDucat}.\ We mostly refer to this approach in the rest of this paper.\ In Subsection \ref{subsect:unproj for index 2} we give a brief overview of the construction in \cite{CampoC4I2}.\

\subsection{Construction} \label{subsect:unproj for index 2}

Let $\bar{X} \subset w\mathbb{P}^7$ be a quasi-smooth codimension 4 $\mathbb{Q}$-Fano 3-fold and $\bar{Z} \subset w\mathbb{P}^6$ a codimension 3 $\mathbb{Q}$-Fano 3-fold, both of index $\iota=1$, and suppose that $\bar{X}$ is obtained as Type I unprojection of $Z$ at a divisor $D \cong w \mathbb{P}^2$ embedded as a complete intersection inside $\bar{Z}$ (for a detailed study of Type I unprojections, see \cite{KustinMiller,PapadakisReidKM,PapadakisComplexes,T&Jpart1}).\ 
For $W$ a set of seven positive non-zero integers, call $x_0,x_1,x_2,y_1,y_2,y_3,y_4,s$ the coordinates of $w\mathbb{P}^7 = \mathbb{P}^7(1,W)$, and consider $\gamma$ the $\Z/2\Z$-action on $w\mathbb{P}^7$ that changes sign to the coordinate of $x_0$ of weight 1, that is,
\begin{equation}
    \gamma \colon \left( x_0, x_1, x_2, y_1, y_2, y_3, y_4, s \right) \mapsto \left( -x_0, x_1, x_2, y_1, y_2, y_3, y_4, s \right) \; .
\end{equation}
Here, the divisor $D$ is defined by the ideal $I_D \coloneqq \langle y_1,\dots,y_4 \rangle$.\ 
Provided that the equations of $\bar{X}$ are invariant under $\gamma$, it is possible to perform the quotient of $\bar{X}$ by $\gamma$.\ 
The 3-fold $X \coloneqq \bar{X}\!/\!\gamma$ obtained as such is Fano, has terminal singularities, is quasi-smooth, its ambient space is $\mathbb{P}^7(2,W)$, and has index $\iota=2$ (cf \cite[Lemmas 3.1, 3.2, 3.3, 3.4]{CampoC4I2}).\ The coordinates of $\mathbb{P}^7(2,W)$ are $\xi,x_1,x_2,y_1,y_2,y_3,y_4,s$, where $\xi \coloneqq x_0^2$.\
This construction can be summarised by the following diagram.\
\begin{equation} \label{index 2 diagram}
\xymatrix{
	\iota=1 & \bar{X} \ar[d]_{\Z/2\Z}^\gamma & & \bar{Z} \ar@{-->}[ll]_{\text{unprojection}} \\
	\iota=2 & X & & }
\end{equation}
The key point is to find an appropriate index 1 double cover $\bar{X}$ for $X$ of index 2.\ The double cover is ramified on the half elephant $\left| -\frac{1}{2} K_X \right|$, and the equations of $X$ are inherited from the ones of $\bar{X}$ (see \cite[Theorem 1.1]{CampoC4I2}).\ 

Recall from \cite{T&Jpart1} that there are between two and four different deformation families for $\bar{X}$ of index 1 sharing the same Hilbert series.\ They are derived from just as many so-called \textit{formats} for the $5\times 5$ antisymmetric graded matrix $M$ whose five maximal Pfaffians determine the equations of $\bar{Z}$.\ 
These formats, defined by specific constraints on the polynomial entries of the matrix $M$ (cf \cite[Definition 2.2]{T&Jpart1}), are called \textit{Tom} and \textit{Jerry} formats.\ Accordingly, $\bar{X}$ is said to be either of \textit{Tom type} or of \textit{Jerry type} (cf \cite[Definition 2.2]{CampoSarkisov}).\ Not all the formats are compatible with the double-cover construction, that is, not all formats descend to index 2.\ 
However, exactly one Tom format always does.\ 
A criterion to determine which formats in index 1 become formats in index 2 via the double-cover construction can be found in \cite[Theorems 4.3, 5.1]{CampoC4I2}; this exhausts all the Tom and Jerry formats.\ 

When a Type I unprojection is employed, we only focus on $\bar{X}$ of Tom type, which represent 32 families out of the 34 we study in this paper.\ In this context, $\bar{X}$ is a general member in its Tom family, provided the $\gamma$-invariance.\ Thus, $X$ is general under the above constraints.\

A close analysis of the unprojection equations of $\bar{X}$, and therefore of $X$ gives crucial insights on the behaviour of the birational links from $X$.\ The unprojection equations of $\bar{X}$ are of the form $s y_i = g_i(x_0, x_1, x_2, y_1, \dots, y_4)$ for $1 \leq i \leq 4$ (cf \cite[Theorem 4.3]{PapadakisComplexes}).\ 
Consequently, four of the nine equations defining $X$ are of the form 
\begin{equation*}
s y_i = g_i(\xi, x_1, x_2, y_1, \dots, y_4) \qquad \text{for }1 \leq i \leq 4 \, ;  
\end{equation*}
with a little abuse of notation, we call them \textit{unprojection equations} as well.\ 
The point $\mathbf{p}_s \in X$ is a Type I centre (\cite[Theorem 3.2]{T&Jpart1} and \cite[Lemma 3.4]{CampoC4I2}).\ Note that the unprojection variable $s$ appears linearly in the unprojection equations of $\bar{X}$ and $X$, hence the point $\mathbf{p}_s \in X$ is called a \textit{linear cyclic quotient singularity} in \cite[Definition 3.28]{GuerreiroThesis}; we will occasionally use this nomenclature in the following.\ The equations of $X$ are therefore of the form
\begin{equation*}
    \left( \Pf_i = sy_j-g_j=0, \,\, 1\leq i \leq 5,\, 1\leq j\leq 4 \right)
\end{equation*}
for $\Pf_i,\,g_j \in \mathbb{C}[\xi,x_1,x_2,y_1,\ldots,y_4]$.\ 
From \cite[Lemma 3.5]{CampoSarkisov} we have that each unprojection equation of the index 1 double cover $\bar{X}$ contains at least one monomial only in the orbinates $x_0,x_1,x_2$.\ The following lemma is a consequence of the  constraint of having $\bar{X}$ invariant under the action $\gamma$.\

\begin{Lem} \label{lem:RHS unproj} 
Consider the four unprojection equations $s y_i = g_i$ of $X$ for $1 \leq i \leq 4$, and suppose that $\wt(x_1)$ is even.\ Then, there are at least three $g_i$ that are of the form $g_i = f_i(\xi,x_1) + h_i(\xi, x_1, x_2, y_1, \dots, y_4)$.\ 
\end{Lem}

\begin{proof}
Since $\iota_X=2$, for terminality reasons the basket of singularities of $X$ consists only of cyclic quotient singularities with odd order (cf \cite[Lemma 1.2 (3)]{SuzukiIndexBound}).\ 
Consequently, the weight of the unprojection variable $s$ is always odd.\ By hypotheses, the weights of $\xi$ and $x_1$ are even, so the orbinate $x_2$ has odd weight.\ 
By direct observation, at least three of the coordinates $y_1,\dots,y_4$ have odd weight.\ 
Note that if $s y_i$ for some $1 \leq i \leq 4$ has odd weight, the corresponding $g_i$ does not contain any pure monomial in $\xi, x_1$; that is, $\wt(y_i)$ must be odd for it to happen.\ 

We briefly recall the notation of unprojection equations necessary to this proof; for the full details of the construction we refer to \cite[Section 5.3]{PapadakisComplexes} and \cite[Appendix]{CampoSarkisov}.\ 
To fix ideas, suppose that the matrix $M$ is in Tom$_1$ format; the proof for the other Tom formats is analogous.\ Such matrix is of the form
\begin{equation*}
	M =\left(
	\begin{array}{c c c c}
	p_1 & p_2 & p_3 & p_4 \\
	& q_1 & q_2 & q_3 \\
	& & q_4 & q_5 \\
	& & & q_6
	\end{array}
	\right) 
\end{equation*}
for $p_i \not\in I_D$ and $q_i \in I_D$ homogeneous polynomials in the given degrees of $M$.\ 
Without loss of generality, we can fill the entries of $M$ with linear monomials when the Tom constraints and the degree prescription on $M$ allow us to do so (for details see \cite[Section 6.2]{T&Jpart1}).\ 
In this context, at least three of the $q_i$ can be filled with one of the $y_1,\dots,y_4$.\ By homogeneity of the Pfaffians, at least two of the $p_i$ have even degree; thus, $\xi$ and $x_1$ can occupy those entries (not necessarily linearly).\ 
Define the matrices $N_j$ as 
\begin{equation*} 
N_j =
\left(
\begin{array}{c c c c}
p_1 & p_2 & p_3 &  p_4 \\
& \alpha^j_{23} & \alpha^j_{24} & \alpha^j_{25} \\
& & \alpha^j_{34} & \alpha^j_{35} \\
& & & \alpha^j_{45}
\end{array}
\right) 
\end{equation*}
where $\alpha^j_{kl}$ is the coefficient of $y_j$ in $q_{kl}$.\ Let $Q$ be the $4 \times 4$ matrix $ Q = \left( \Pf_i(N_j) \right)_{i,j=1 \dots 4}$, where $\Pf_i(N_j)$ is calculated by excluding the $(i + 1)$-th row and column of $N_j$ for $i=1, \dots 4$.\
The polynomials $g_i$ in the right-hand side of the unprojection equations are given by $g_i = \frac{1}{p_i} \hat{Q}_i$, where $\hat{Q}_i$ is the $3 \times 3$ matrix obtained deleting the $i$-th row and column of $Q$.\

Since some of the polynomials $q_i$ are quasi-linear in the $y_i$, the polynomials $p_i$ having even degrees are multiplied by 1 at least twice in the Pfaffians of the matrices $N_j$ (also note that the matrices we consider have always weights as in \cite[Equation (4.2) and (4.3)]{CampoSarkisov}).\ 
Thus, at least three entries of each row of the $4\times 4$ matrix $Q$ contain a monomial purely in $\xi, x_1$.\ 
So, at least two of these entries are included in a $3\times 3$ minor of $Q$.\ 
Then, by taking the determinant of $\hat{Q}_i$ we have that $g_i$ has a monomial purely in $\xi, x_1$ up to a $p_i$ factor, which gets simplified in the definiton of $g_i$ thanks to \cite[Lemma 5.3]{PapadakisComplexes}.\ 
At worst, the two entries of $\hat{Q}_i$ containing pure monomials in $\xi, x_1$ are all concentrated in a $2\times 3$ block (or two $1\times 3$ blocks).\ 
This implies that only three of the $g_i$ have the desired monomials $f_i(\xi,x_1)$.\ Otherwise, all $g_i$ contain $f_i(\xi,x_1)$.\
\end{proof}
By Lemma \ref{lem:RHS unproj} we have that the four unprojection equations of $X$ are of the form 
\begin{equation} \label{eqn: RHS unproj}
\begin{cases}
sy_i = f_i(\xi,x_1)+h_i(\xi,x_1,x_2,y_1\ldots,y_4),\quad 1 \leq i \leq 3 \\
sy_4 = h_4(\xi,x_1,x_2,y_1\ldots,y_4)   
\end{cases}
\end{equation}
where $f_i$ is not identically zero by quasi-smoothness of $X$ and $\wt(y_4)$ is even.\

\begin{Lem} \label{lem: f_i alg independent}
Suppose that $\wt(y_4)$ is even.\ Then, the polynomials $f_i(\xi,x_1)$ for $1 \leq i \leq 3$ are algebraically independent.\
\end{Lem}

\begin{proof}
We refer to the notation introduced in Lemma \ref{lem:RHS unproj} and in Equation \eqref{eqn: RHS unproj}.\ 
Consider the equations of $X$ in Equation \eqref{eqn: RHS unproj} and evaluate them at $(y_i = x_2 =0)$ for $1 \leq i \leq 4$.\ Hence we get the system of polynomial equations $f_i(\xi,x_1) = 0$ for $1 \leq i \leq 3$.\
Now suppose that the $f_i(\xi,x_1)$ are not algebraically independent: they therefore have a common solution, that is, a finite set of points in $X$.\ Such points are quotient singularities with even order.\ This is not possible as $X$ cannot have singularities with even order (cf \cite[Lemma 1.2 (3)]{SuzukiIndexBound}), as it would not be quasi-smooth.\ Thus, $f_i(\xi,x_1)$ for $1 \leq i \leq 3$ are algebraically independent.\
\end{proof}

The remaining two families (GRDB ID \#39569 and \#39607) that we investigate here are constructed in a similar fashion as in diagram \eqref{index 2 diagram}: in these cases, their respective double covers $\bar{X}$ only have Type II$_2$ centres, or worse (see \cite[Section 1]{PapadakisEqnsTypeII1}, \cite[Definition 2.2, Theorem 2.15]{PapadakisTypeIIUnproj} for Type II unprojections, and \cite[Section 3.6]{T&Jpart1}, \cite{ReidKinosaki} for Type II centres).\ Hence, the Tom and Jerry formats are not applicable as $\bar{X}$ is obtained via Type II$_2$ unprojections from a Fano hypersurface: see \cite[Section 7]{CampoC4I2} for the construction of these double covers.\

Among the Hilbert series listed in the GRDB there is also the smooth \#41028; there are two distinguished deformation families associated to this Hilbert series: these are the classical examples of a divisor of bidegree $(1,1)$ in $\mathbb{P}^2 \times \mathbb{P}^2 \subset \mathbb{P}^8$ and $\mathbb{P}^1 \times \mathbb{P}^1 \times \mathbb{P}^1 \subset \mathbb{P}^7$ (cf \cite[Section 2]{T&Jpart1}).\ These are rational, and therefore non-solid.\ This confirms the statement of Theorem \ref{thm:main} also in the case of~\#41028.\

\subsection{Lift under the Kawamata blowup} \label{subsect: Kawamata blowup}

We want to initiate a birational link by blowing up a cyclic quotient singularity on $X$.\ In order to understand what the equations of the blowup $Y$ are, we first explain how the sections in $H^0(X,mA)$ lift to $Y$ under a Kawamata blowup of $\mathbf{p_s} \in X$, for $m\geq 1$.\

Recall that by \cite[Theorem 4.3]{PapadakisComplexes}, \cite[Theorem 3.2]{T&Jpart1}, and \cite[Lemma 3.4]{CampoC4I2}, $\mathbf{p_s}$ is a linear cyclic quotient singularity of $X$ as in \cite[Definition 2.6.1]{GuerreiroThesis}; locally, $\mathbf{p_s} \sim \frac{1}{a_s}(a_0,a_1,a_2)$.\ In the notation introduced in Subsection \ref{subsect:unproj for index 2}, $a_0,a_1,a_2$ are the weights of the orbinates $\xi,x_1,x_2$.\

We can assume (cf \cite[Lemma 1.2 (3)]{SuzukiIndexBound}) that $a_0=2$ is equal to the Fano index of $X$.\ Moreover, since $\mathbf{p_s}$ is terminal, $\gcd(a_s,a_0a_1a_2)=1$ and, in particular, there is $k \in \mathbb{Z}$ such that $ka_0 \equiv 1 \pmod{a_s}$.\ Denote by $\overline{a} < a_s$ the unique remainder of $a \bmod{a_s}$.\ Then,
\begin{equation*}
\mathbf{p_s} \sim \frac{1}{a_s}(1,\overline{ka_1},\overline{ka_2}) \sim \frac{1}{a_s}(1,\overline{ka_1},a_s-\overline{ka_1})=\frac{1}{a_s}\bigg(1,\frac{a_1}{2},a_s-\frac{a_1}{2}\bigg) \; .\
\end{equation*}

\begin{Lem} \label{lem:lift}
Let $\varphi \colon Y \rightarrow X$ be the Kawamata blowup centred at $\mathbf{p_s}$.\ Then 
\begin{align*}
&s \in H^0(Y,-m_1K_Y+m_2E) && \text{for some $m_1,\,m_2>0$;}\\
&\xi \in H^0(Y,-K_Y) , \qquad
x_1 \in H^0(Y,-mK_Y) && \text{for some $m>0$;}\\
&z_i \in H^0(Y,-m_iK_Y-n_iE) && \text{for some $m_i,\,n_i>0$}
\end{align*}
where $z_i$ is one of $x_2,\, y_1,\,y_2,\,y_3,\,y_4$.
\end{Lem}

\begin{proof}
Recall that for the Kawamata blowup $\varphi \colon Y \rightarrow X$ we have $K_Y=\varphi^*(K_X)+\frac{1}{r}E $, where $\frac{1}{r}$ is its discrepancy and $r$ is the index of the cyclic quotient singularity.\ Let $x$ be one of the sections above and $\nu$ be its vanishing order at $E$.\ Then,
\begin{align*}
x &\in H^0\left(Y,\frac{a_i}{2}\varphi^*(-K_X)-\nu E \right) \\ 
    &= H^0\left(Y,\frac{a}{2}(-K_Y+\frac{1}{a_s}E)-\nu E \right)\\
    &= H^0\left(Y,-\frac{a}{2}K_Y+\frac{a-2a_s\nu}{2a_s}E \right) \; .
\end{align*}

By the description of $\mathbf{p_s}$ it follows that $\xi$ vanishes at $E$ with order $\nu=\frac{1}{a_s}$.\ Similarly, the vanishing order of $x_1$ and $x_2$ at $E$ is $\nu=\frac{a_1}{2a_s}$ and $\nu=\frac{2a_s-a_1}{2a_s}$.\ Hence, 
\begin{equation*}
\xi \in H^0(Y,-K_Y), \quad x_1 \in H^0\left(Y,-\frac{a_1}{2}K_Y \right), \quad x_2 \in H^0\left(Y,-\frac{a_2}{2}K_Y-\frac{1}{2}E \right).\
\end{equation*}

On the other hand, $ s \in H^0(Y,\frac{a_s}{2}\varphi^*(-K_X))= H^0(Y,-\frac{a_s}{2}K_Y+\frac{1}{2}E)$.\ 

Starting from \eqref{eqn: RHS unproj}, we compute the vanishing order of $y_i,\, 1 \leq i\leq 3$ at $E$.\ The monomials of $f_i(\xi,x_1)$ are of the form $\xi^{\alpha}x_1^{\beta}$ and if $d_i$ is the homogeneous degree of the equation $sy_i-g_i=0$, we have $\alpha a_0 + \beta a_1 = d_1$.\ On the other hand such monomials vanish at $E$ with order
\begin{equation*}
\alpha\cdot \frac{1}{a_s} + \beta \cdot \frac{a_1}{2a_s} = \frac{d_i}{2a_s} < 1 \;.
\end{equation*}
Hence, $\xi^{\alpha}x_1^{\beta}$ is pulled back by $\varphi$ to $\xi^{\alpha}x_1^{\beta}t^{d_i/2a_s}$ where $E \coloneqq (t =0)$ is the exceptional divisor.\ Since $\frac{d_i}{2a_s}<1$, when saturating the ideal of $Y$ with respect to $t$, we find that $\xi^{\alpha}x_1^{\beta}t^{d_i/2a_s}$ becomes $\xi^{\alpha}x_1^{\beta}$.\ Hence $sy_i-g_i = 0$ is a divisor in $|-\frac{d_i}{2}K_Y|$ and we conclude that 
$
y_i \in H^0\left(Y, -\frac{d_i-a_s}{2}K_Y-\frac{1}{2}E\right)
$.\

For $y_4$, since $sy_4-h_4=0$ contains no pure monomials in $\xi,x_1$ we can only say that $y_4$ vanishes at $E$ with order $\nu$ at least $\frac{\wt(y_4)+a_s}{2a_s}$; therefore $\frac{\wt(y_4)-2a_s\nu}{2a_s} \leq -\frac{1}{2}$ with equality when the vanishing order is exactly $\frac{\wt(y_4)+a_s}{2a_s}$.\ In other words,
\begin{equation*}
y_4 \in  H^0\left(Y,-\frac{\wt(y_4)}{2}K_Y-m_4E \right), \quad m_4 \geq \frac{1}{2}>0 \; . \qedhere
\end{equation*}
\end{proof}

The Lemma below follows from Lemma \ref{lem:lift} and Lemma \ref{lem:RHS unproj}.\
\begin{Lem} \label{lem:eqnsY}
Let $X \subset w\mathbb{P}^7$ be defined by the equations $(\Pf_i = sy_j-g_j=0, \,\, 1\leq i \leq 5,\, 1\leq j\leq 4) $, 
where $\Pf_i,\,g_j \in \mathbb{C}[\xi,x_1,x_2,y_1,\ldots,y_4]$.\ 
Then, the Kawamata blow up $Y$ of $X$ at $\mathbf{p_s} \in X$ is defined by equations $(\Pf_i = sy_j-g_j=0, \,\, 1\leq i \leq 5,\, 1\leq j\leq 4) $, 
with $\Pf_i,\,g_j \in \mathbb{C}[t,\xi,x_1,x_2,y_1,\ldots,y_4]$, where $sy_j-g_j \in |-m_jK_Y|$ for exactly three values of $j$ and $\Pf_i\in |-m_iK_Y-n_iE|$, with $n_i>0$.\ 
\end{Lem}

\section{Birational Links and Mori Dream Spaces} \label{sec:links and mds}

We use the language of Mori Dream Spaces to study the birational links discussed later in this Section.\ 
We recall the definition of Mori Dream Space as in \cite{mdsGIT} and some of their properties within the context of the Sarkisov Program, as explained in  \cite{Corti95,hamidquartic,hamidmds,2raybrown}.\
This section is devoted to give a theoretical background to show that we can perform the 2-ray game for codimension 4 index 2 Fano 3-folds.\

\begin{Def}[{\cite[Definition~1.10]{mdsGIT}}]
A normal projective variety $Z$ is a \textit{Mori Dream Space} if the following hold
\begin{itemize}
	\item $Z$ is $\mathbb{Q}$-factorial and $\Pic(Z)$ is finitely generated;
	\item $\Nef(Z)$ is the affine hull of finitely many semi-ample line bundles;
	\item There exists a finite collection of small $\mathbb{Q}$-factorial modifications $f_i \colon Z \dashrightarrow Z_i$ such that each $Z_i$ satisfies the previous point and $\overline{\Mov}(Z)= \bigcup f_i^*(\Nef(Z_i))$.\
\end{itemize}
\end{Def}

\begin{Rem}
As pointed out in \cite[Remark~2.4]{OkawaMDS}, if we work over a field which is not the algebraic closure of a finite field, then the condition that $\Pic(Z)$ is finitely generated is equivalent to $\Pic(Z)_{\mathbb{R}} \simeq N^1(Z)_{\mathbb{R}}$.\  
\end{Rem}

In characteristic zero, it is known that any klt pair $(Z, \Delta)$ where $Z$ is $\mathbb{Q}$-factorial and $-(K_Z+\Delta)$ is ample is a Mori Dream Space (see \cite[Corollary~1.3.2]{BCHM}).\ Examples include weak Fano varieties.\  

We have the following lemma.\

\begin{Lem}[{\cite[Proposition~1.11 (2)]{mdsGIT}}]
Let $Z$ be a Mori Dream Space.\ Then, there are finitely many birational contractions $g_i \colon Z \dashrightarrow Z_i$ where for each $i$, $Z_i$ is a Mori Dream Space and
\begin{equation*}
\overline{\Eff}(Z)= \bigcup_i \mathcal{C}_i,\quad \mathcal{C}_i =  g_i^*(\Nef(Z_i))+\mathbb{R}_+[E_1]+ \cdots +\mathbb{R}_+[E_k], 
\end{equation*}
where $E_1,\ldots, E_k$ are the prime divisors contracted by $g_i$.\ 
If $Z_i$ and $Z_j$ are in adjacent chambers, then they are related by a small $\mathbb{Q}$-factorial modification.\
\end{Lem}

We have the following result:
\begin{Lem}[{\cite[Lemma~2.9]{hamidquartic}}]
Let $X$ be a $\mathbb{Q}$-Fano 3-fold and $\varphi \colon Y \rightarrow X$ be a divisorial extraction.\ Then $\varphi$ initiates a Sarkisov link if and only if the following hold:
\begin{enumerate}
    \item $Y$ is a Mori Dream Space;
    \item If $\tau \colon Y \dashrightarrow Y'$ is a small birational map and $Y'$ is $\mathbb{Q}$-factorial, then $Y'$ is terminal;
    \item $[-K_Y] \in \Int(\overline{\Mov}(Y))$.\
\end{enumerate}
\end{Lem}

It is not true that the blowup of a Mori Dream Space is a Mori Dream Space (see \cite{CastravetMDS} for many examples).\ However, the Kawamata blowup $\varphi \colon Y \rightarrow X$ of a $\mathbb{Q}$-Fano 3-fold centred at a linear cyclic quotient singularity is at worst a weak Fano 3-fold with $\mathbb{Q}$-factorial terminal singularities and $[-K_Y] \in \Int(\overline{\Mov}(Y))$.\ 
In this case, there are at least two contractions from $Y$.\ One is a Mori contraction $\varphi \colon Y \rightarrow X$ and the other one is the small contraction associated to the linear system $|-K_Y|$.\ We prove that the small $\mathbb{Q}$-factorial modification associated to the latter is an isomorphism. %
%
The only small $\mathbb{Q}$-factorial modifications which are not isomorphisms are flips.\ This allows us to always remain in the Mori category since the discrepancies increase (see \cite[Lemma~3.38]{KollarMori}).\

\paragraph{Biratonal links and Toric Varieties.\ } Let $T$ be a rank 2 toric variety (up to isomorphism) for which the toric blowup $ \Phi \colon T \rightarrow \mathbb{P}$ restricts to the unique Kawamata blowup $Y \rightarrow X$ centred at $\mathbf{p_s} \in X$.\ 
Then, $\Cl(T)= \mathbb{Z}[\Phi^*H]+\mathbb{Z}[E]$, where $H$ is the generator of the Class Group of $\mathbb{P}$ and $E = \Phi^{-1}(\mathbf{p_s})$ is the exceptional divisor.\ Notice that $\mathbf{p_s}$ is not in the support of $H$.\
Then, the Cox ring of $T$ is
\begin{equation*}
\Cox(T) = \bigoplus_{(m_1,m_2) \in \mathbb{Z}^2} H^0(T,m_1\Phi^*H+m_2E) \; .
\end{equation*}

Since $T$ is toric, the Cox ring of $T$ is isomorphic to a (bi)-graded polynomial ring, in this case,
\begin{equation*}
\Cox(T) \simeq \mathbb{C}[t,\xi,x_1,x_2,y_1\ldots,y_4,s] \; 
\end{equation*}
where the grading comes from the $\mathbb{C}^* \times \mathbb{C}^*$-action defining $\mathbb{P}$ and $\Phi$, respectively. We denote by $\mathbb{R}_+\Big[\Big(\begin{smallmatrix}
  \omega_1 \\
  \omega_2 
\end{smallmatrix}\Big)\Big]$ the ray generated by a divisor $D$ in the linear system $\Big|\mathcal{O}_T\Big(\begin{smallmatrix}
  \omega_1 \\
  \omega_2 
\end{smallmatrix}\Big)\Big|$.\ Over $N^1(T)_{\mathbb{R}} \simeq \mathbb{R}^2$, we can depict the rays generated by the sections of Lemma \ref{lem:lift} as in Figure \ref{fig:chambdecomplin}.\
The movable and effective cone of $T$ in $N^1(T)_{\mathbb{R}}$ are 
\begin{equation*}
\Mov(T)=\mathbb{R}_+[M_1]+\mathbb{R}_+[M_2] \subset \Eff(T)= \mathbb{R}_+[E]+\mathbb{R}_+[E'] \; .
\end{equation*}
Notice that since $T$ is a normal simplicial toric variety, it follows from \cite[Theorem~15.1.8 and Theorem 15.1.10]{Cox}, respectively, that both $\Eff(T)$ and $\Mov(T)$ are closed in the Euclidean topology. According to Lemma \ref{lem:lift}, notice that the rays $E,\, M_1$ and $-K_Y$ cannot coincide.\ On the other hand, it can happen that some of the other rays do coincide.\

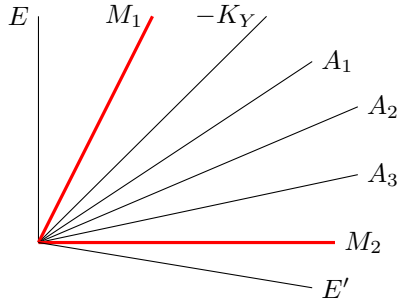
\begin{figure}[h]
\centering
\begin{tikzpicture}[scale=3]
  \coordinate (A) at (0, 0);
  \coordinate [label={left:$E$}] (E) at (0, 1);
  \coordinate [label={left:$-K_Y$}] (K) at (1, 1);

\coordinate [label={left:$M_1$}] (5) at (0.5,1);
\coordinate [label={right:$A_1$}] (A1) at (1.2,0.8);
\coordinate [label={right:$A_2$}] (A2) at (1.4,0.6);
\coordinate [label={right:$A_3$}] (A3) at (1.4,0.3);

\coordinate [label={right:$M_2$}] (M2) at (1.3,0);
\coordinate [label={right:$E'$}] (E') at (1.2,-0.2);

    \draw  (A) -- (E);
    \draw  (A) -- (K);
	\draw [very thick,color=red] (A) -- (5);
	\draw (A) -- (A1);
	\draw (A) -- (A2);
	\draw (A) -- (A3);

    \draw [very thick,color=red] (A) -- (M2);
	\draw (A) -- (E');
\end{tikzpicture}
\caption{A representation of the Mori chamber decomposition of $T$.\ The outermost rays generate the cone of pseudo-effective divisors of $T$ and in red it is represented the subcone of movable divisors of $T$.}%
\label{fig:chambdecomplin}%
\end{figure}

We run several 2-ray games on $T$.\
We divide the construction of the elementary Sarkisov links in two main cases, depending roughly on the Mori chamber decomposition of $T$ in Figure \ref{fig:chambdecomplin}, namely the behaviour of its movable cone of divisors near the boundary of the effective cone of divisors of $T$.\ We consider three cases:
\begin{enumerate}
    \item \label{item: fibr} \textbf{Fibration}: the class of $M_2$ is linearly equivalent to a rational multiple of $E'$.\
    \item \label{item: div contr} \textbf{Divisorial Contraction}: the class of $M_2$ is not linearly equivalent to a rational multiple of $E'$.\ Moreover, we contract the divisor $E'$:
    \begin{enumerate}
        \item \label{item: contr to point} \textbf{to a point}: no generator class of the Cox ring of $T$ is linearly equivalent to a rational multiple of $M_2$, or 
        \item \label{item: contr to rat curve} \textbf{to a rational curve}: there is one (and only one) generator class of the Cox ring of $T$ which is linearly equivalent to a rational multiple of $M_2$.\
    \end{enumerate}
\end{enumerate}

The diagram below sets the notation used in the rest of the paper.\ Note that we only consider Fano 3-folds not in Table \ref{tab:dimlinsys}.\ In each case we have a birational link at the level of toric varieties
\begin{equation*}
\begin{tikzcd} 
& T \arrow[swap]{dl}{\Phi} \arrow[swap]{dr}{\alpha_0} \arrow[dashed]{rr}{\tau_0} & &  T_1\arrow[swap]{dr}{\alpha_1} \arrow{dl}{\beta_0} \arrow[dashed]{r}{\tau_1} & \cdots \arrow[dashed]{r}{\tau_n}   & T'\arrow{dr}{\Phi'}\arrow{dl}{\beta_{n-1}}  & \\
 \mathbb{P} &  & \mathcal{F}_0 &  & \cdots  &  & \mathcal{F}'  
\end{tikzcd}
\end{equation*}
and $\dim \mathcal{F}' \leq \dim T'$, with equality if and only if we are in the second case.\ We restrict the diagram above to the Kawamata blowup $ \varphi \colon Y \subset T \rightarrow X \subset \mathbb{P}$ in order to obtain a birational link between 3-folds.\ 

It follows from \cite[Proposition~2.11]{mdsGIT}, that the birational contractions of a Mori Dream Space are induced from Toric Geometry.\ 
By \cite[Proposition~1.11]{mdsGIT}, one can always carry out a classical Mori Program for any divisor on a Mori Dream space $Y$.\ In particular, when $\rho(Y)=2$ it is called a \textit{2-ray game}.\ We refer the reader to \cite{Corti2ray} for the precise definition of 2-ray game.\ 

The next three lemmas describe the nature of the restriction of the maps $\tau,\tau_i$ to the birational link relative to $X \subset \mathbb{P}$.\

\begin{Lem} \label{lem:iso}
The map $\tau_0$ restricts to an isomorphism on $Y$.\
\end{Lem}

\begin{proof}
In this proof, call $z_i$ all variables that are not $t,s,\xi,x_1$.\ 
The small modification $\tau_0 \colon T \dashrightarrow T_1$ can be decomposed as

\begin{equation*}
    \xymatrix{
    T \ar@{-->}[rr]^\tau \ar[dr]_{\alpha_0} && T_1 \ar[dl]^{\beta_0} \\
    & \mathcal{F}_0 &
    }
\end{equation*}
where $\mathcal{F}_0 \coloneqq \Proj \bigoplus_{m\geq 1} H^0(T,m\mathcal{O}_T\big(\begin{smallmatrix}
  2 \\
  1 
\end{smallmatrix}\big))$, in coordinates the map $\alpha_0$ is,
\begin{align*}
\alpha_0 \colon T &\rightarrow \mathcal{F}_0 \\
(t,s,\xi,x_1, \ldots z_i \ldots) &\mapsto (\xi,x_1, \ldots u_j \ldots)  
\end{align*}
and $u_j$ is some monomial which is a multiple of $z_i$ and of either $t$ or $s$.\ 

Notice that the irrelevant ideal of $T$ is $(t,s) \cap (\xi, x_1, \ldots z_i, \ldots)$.\ Hence, $\alpha_0$ contracts the locus $(z_i = 0) \subset T$.\ This is indeed a small contraction since the ray $\mathbb{R}_+\Big[\Big(\begin{smallmatrix}
  2 \\
  1 
\end{smallmatrix}\Big)\Big]$ is in the interior of $\Mov(T)$.\ Now we restrict this small contraction to $Y$.\ 
By Lemma \ref{lem:lift}, we know that $z_i$ are the sections in $H^0\Big(Y,-\frac{a_i}{2}K_Y-n_iE\Big)$ with $n_i > 0$.\ On the other hand, by Lemma \ref{lem:eqnsY}, the Pfaffian equations $\Pf_j$ must vanish identically and the remaining equations are $f_j(\xi,x_1) = 0$ for $1 \leq j \leq 3$.\ 
However, this is empty by Lemma \ref{lem: f_i alg independent}.\ 
\end{proof}

\begin{Rem}
It is interesting to observe that the behaviour of the restriction of $\tau_0$ as in Lemma \ref{lem:iso} is not a feature of the codimension of $X$ but rather of its Fano index.\ When the index is 1, the map $\tau$ restricts to a number of simultaneous Atyiah flops (see \cite[Theorem 4.1]{CampoSarkisov} and \cite{okadaI}).\ On the other hand, for higher indices it is an isomorphism (cf \cite[Theorem 2.5.6]{GuerreiroThesis}).\
\end{Rem}

\begin{Lem} \label{lem:wall2}
Suppose that the map $\tau_i \colon T_i \dashrightarrow T_{i+1}$ restricts to a small $\mathbb{Q}$-factorial modification over a point which is not an isomorphism.\ Then, it restricts to a flip.\
\end{Lem}

\begin{proof} 
By assumption there are curves $C_i \subset Y_i$ and $C_{i+1} \subset Y_{i+1}$ such that the diagram

\begin{equation*}
    \xymatrix{
    C_i \subset Y_i \ar@{-->}[rr]^{\tau_i} \ar[dr]_{\alpha_i} && Y_{i+1} \supset C_{i+1} \ar[dl]^{\beta_i} \\
    & \mathcal{F}_i &
    }
\end{equation*}
is a small $\mathbb{Q}$-factorial modification.\ Clearly, we have $K_{Y_i}\cdot C_i < 0$.\ Indeed, by Lemma \ref{lem:eqnsY}, $C_i$ intersects $-K_{Y_i}$ transversely since $C_i \in \mathrm{Bs}|-m_i K_{Y_i}-n_iE|$ for some $m_i,\,n_i > 0$; then, the claim follows.\ On the other hand, there exists a divisor class
\begin{equation*}
L \sim -m_1K_Y-m_2E \, \in \, \mathbb{R}_+[A_i] + \mathbb{R}_+[E']
\end{equation*}
for which $L \cdot C_{i+1}>0$, 
where $A_i$ and $E'$ are as in Figure \ref{fig:chambdecomplin} (implying, in particular, $m_i>0$).\ Moreover,
\begin{equation*}
A_i \sim -n_1K_Y-n_2E,\,\, n_i >0
\end{equation*}
and $A_i \cdot C_{i+1} = 0$.\ Then, $m_2A_i-n_2L \sim (m_2n_1-m_1n_2)(-K_{Y_{i+1}})$.\ 
Notice that $m_2n_1-m_1n_2=0$ if and only if $A_i \sim mL$ for some non-zero $m \in \mathbb Q$, which is not possible since they have different intersections with $C_{i+1}$.\ Additionally, $m_2n_1-m_1n_2>0$ since $L$ is in the cone $\mathbb{R}_+[A_i]+\mathbb{R}_+[E']$ but is not linearly equivalent to any non-zero rational multiple of $A_i$.\ In particular,
\begin{equation*}
-K_{Y_{i+1}} \cdot C_{i+1} = \frac{1}{m_2n_1-m_1n_2}(m_2A_i\cdot C_{i+1}-n_2L\cdot C_{i+1})  <0
\end{equation*}
since $A_i$ and $L$ were chosen so that $A_i\cdot C_{i+1} = 0$ and $L\cdot C_{i+1}>0$.\
\end{proof}

\begin{Lem}[{\cite[Lemma~2.5.7]{GuerreiroThesis}}] \label{lem:discr}
Let $\sigma \colon X \dashrightarrow X'$ be an elementary birational link between $\mathbb{Q}$-Fano 3-folds initiated by a divisorial extraction $\varphi \colon E \subset Y \rightarrow X$.\ Then, there is a birational map $\Psi \colon Y \dashrightarrow X'$ which is the composition of small $\mathbb{Q}$-factorial modifications followed by a divisorial contraction $\varphi' \colon E' \subset Y' \rightarrow X'$ with discrepancy
\begin{equation*}
a= \frac{m_2}{m_1n_2-m_2n_1}
\end{equation*}
where $m_i$ and $n_i$ are positive rational numbers such that $\varphi'^*(-K_{X'}) \sim -m_1K_Y-m_2E$ and $E' \sim -n_1K_Y-n_2E$.\
\end{Lem}

In practice, to successfully run this game we need to guarantee that each step $\tau_i \colon T_i \dashrightarrow T_{i+1}$ of the birational link contracts finitely many curves, with the exception of $\tau_i$ an isomorphism on $Y_i \subset T_i$.\ However, in each case it is possible to retrieve explicitly the loci contracted and extracted by the maps $\tau_i$.\

In the rest of this section we will present several Tables containing information about the links for each family examined.\ We specify the Type I centre whose blowup initiates the link in cases where there is more than one Type I centre; for instance, we write "\#39993 $1/5$" instead of just "\#39993" if the link starts with the Kawamata blowup of a $1/5$ singularity.\

\subsection{Case I: Fibrations}

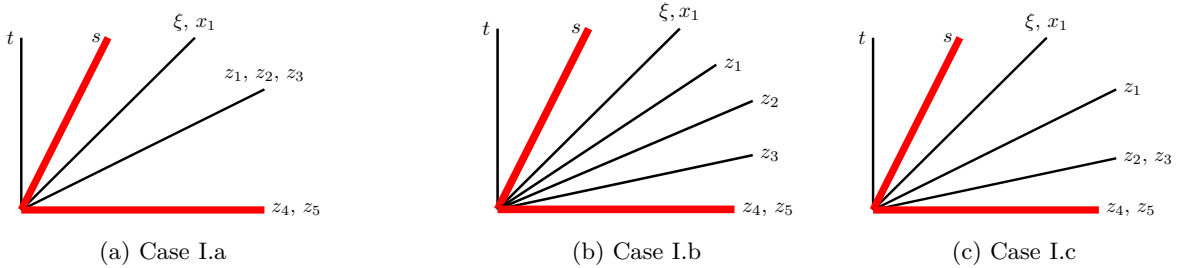
\begin{figure*}[h!]
        \centering
        \begin{subfigure}[b]{0.3\textwidth}
            \centering
            \resizebox{\linewidth}{!}{
\begin{tikzpicture}[every node/.style={scale=0.333}]
  \coordinate (A) at (0, 0);
  \coordinate [label={left:$t$}] (E) at (0, 1);
  \coordinate [label={above:$\xi,\,x_1$}] (K) at (1, 1);

\coordinate [label={left:$s$}] (5) at (0.5,1);

\coordinate [label={above:$z_1,\,z_2,\,z_3$}] (A2) at (1.4,0.7);
\coordinate [label={right:$z_4,\,z_5$}] (A3) at (1.4,0);

  \draw  (A) -- (E);
  \draw  (A) -- (K);
	\draw [very thick,color=red] (A) -- (5);
	\draw (A) -- (A2);
	\draw [very thick,color=red] (A) -- (A3);

\end{tikzpicture}
}
            \caption[Case I.a]
            {{\small Case I.a}}    
            \label{fig:Ia}
        \end{subfigure}
        \hfill
        \begin{subfigure}[b]{0.3\textwidth}
            \centering 
\resizebox{\linewidth}{!}{
\begin{tikzpicture}[every node/.style={scale=0.333}]
  \coordinate (A) at (0, 0);
  \coordinate [label={left:$t$}] (E) at (0, 1);
  \coordinate [label={above:$\xi,x_1$}] (K) at (1, 1);

\coordinate [label={left:$s$}] (5) at (0.5,1);
\coordinate [label={right:$z_1$}] (A1) at (1.2,0.8);
\coordinate [label={right:$z_2$}] (A2) at (1.4,0.6);
\coordinate [label={right:$z_3$}] (A3) at (1.4,0.3);
\coordinate [label={right:$z_4,\,z_5$}] (M2) at (1.3,0);

  \draw  (A) -- (E);
  \draw  (A) -- (K);
	\draw [very thick,color=red] (A) -- (5);
	\draw (A) -- (A1);
	\draw (A) -- (A2);
	\draw  (A) -- (A3);
    \draw [very thick,color=red](A) -- (M2);

\end{tikzpicture}
    }    
            \caption[]
            {{\small Case I.b}}    
            \label{fig:Ib}
        \end{subfigure}
        \begin{subfigure}[b]{0.3\textwidth}   
            \centering 
            \resizebox{\linewidth}{!}{
\begin{tikzpicture}[every node/.style={scale=0.333}]
  \coordinate (A) at (0, 0);
  \coordinate [label={left:$t$}] (E) at (0, 1);
  \coordinate [label={above:$\xi,\,x_1$}] (K) at (1, 1);

\coordinate [label={left:$s$}] (5) at (0.5,1);
\coordinate [label={right:$z_1$}] (A2) at (1.4,0.7);
\coordinate [label={right:$z_2,\,z_3$}] (A3) at (1.4,0.3);
\coordinate [label={right:$z_4,\,z_5$}] (M2) at (1.3,0);

  \draw  (A) -- (E);
  \draw  (A) -- (K);
	\draw [very thick,color=red] (A) -- (5);
	\draw (A) -- (A2);
	\draw (A) -- (A3);
    \draw [very thick,color=red] (A) -- (M2);
\end{tikzpicture}
}
            \caption[]
            {{\small Case I.c}}    
            \label{fig:Ic}
        \end{subfigure}
        \hfill
        \caption{\small The possible effective cones of $T$ ending with a fibration to a rational curve $B' = \Proj \mathbb{C}[z_4,z_5]$.\ The variables $z_1,\ldots,z_5$ are $y_1,\ldots,y_4,x_2$ up to permutation.}
        \label{fig:fibr}
    \end{figure*}

In each case in Figure \ref{fig:fibr}, whose notation we refer to, the movable cone of $T$ is not strictly contained in the effective cone of $T$.\ Indeed, the rays generated by $z_4$ and $z_5$ are both in $\partial \Mov(T)$ and $\partial \Eff(T)$.\ Hence we have a diagram of toric varieties
\begin{equation*}
\begin{tikzcd} 
T\arrow[swap]{dd}{\Phi} \arrow[dashed]{rr}{\tau}& &  T'\arrow{dd}{\Phi'}\\
&&\\
 \mathbb{P} \arrow[swap,dashed]{rr}{\sigma} &  &\mathcal{F}'  
\end{tikzcd}
\end{equation*}
where $\Phi$ is a divisorial contraction, $\Phi'$ is a fibration into $\mathcal{F}'\simeq \Proj\mathbb{C}[z_4,z_5] \simeq \mathbb{P}^1$.\ The map $\tau$ is a small $\mathbb{Q}$-factorial modification.\  
We restrict $\Phi \colon T \rightarrow \mathbb{P}$ to be the unique Kawamata blowup $\varphi \colon E \subset Y \rightarrow \mathbf{p_s} \in X$.\ 
By Lemmas \ref{lem:iso} and \ref{lem:wall2}, the map $\tau|_Y \colon Y \dashrightarrow Y' $ is an isomorphism followed by a finite sequence of isomorphisms or flips.\ 
Referring to the notation in Figure \ref{fig:pliconj}, by assumption the rays $\mathbb{R}_+[z_4]$ and $\mathbb{R}_+[z_5]$ are given by linearly dependent vectors $v_4$ and $v_5$ in $\mathbb{Z}^2$.\ Let $B$ be the matrix whose columns are the vectors $v_4$ and $v_5$.\ Then, there is $A \in \GL(2,\mathbb{Z})$ such that 
\begin{equation*}
A\cdot B = \begin{pmatrix}
a & b \\
0 & 0 
\end{pmatrix} \; .
\end{equation*}
Then, performing a row operation on the grading of $T'$ via the matrix $A$, $T'$ is isomorphic to a toric variety with $\mathbb{C}^* \times \mathbb{C}^*$-action given by
\begin{equation*}
\begin{array}{ccccccccccccc}
             &       & t  & s &   \xi & x_1 & z_1 & z_2 & z_3 & z_4 & z_5 \\
\actL{T'}   &  \lBr &  \kappa_0 & \kappa_1 &   \kappa_2 & \kappa_3 & \kappa_4 & \kappa_5 & \kappa_6 & a & b &   \actR{}\\
             &       & \lambda_0 & \lambda_1 &  \lambda_2 & \lambda_3 & \lambda_4 & \lambda_5 & \lambda_6 & 0 & 0 &  
\end{array}
\end{equation*}

where the column vectors are ordered clockwise.\ The irrelevant ideal of $T'$ is $(t , s,   \xi, x_1, z_1, z_2, z_3) \cap (z_4, z_5)$.\ The map $\Phi'$ can be realised as 
\begin{align*}
\Phi' \colon T' &\longrightarrow \mathcal{F}' \\
(t , s,  \xi, x_1, z_1, z_2, z_3, z_4, z_5) & \longmapsto (z_4,z_5)
\end{align*}
and the fibre of $\Phi'$ over each point is isomorphic to $\mathbb{P}(\lambda_0,\ldots,\lambda_6)$.\ The following Table \ref{tab:fibrations} shows what the restrictions of $\Phi'$ to $Y'$ are.\

\begin{table}[ht!]
    \centering
    \begin{tabular}[t]{l c c c}
        \toprule
        ID   & Centre & $T'$             &  $dP_k/\mathcal{F}' $     \\ \midrule

\#39993 & 1/5   & $\begin{pmatrix}
t & s & \xi & x_1 &y_1 & y_2 & y_3 & x_2 & y_4\\
0 & 5 & 2 & 4 & 3 & 3 & 3 & 1 & 2\\
1 & 3 & 1 & 2 & 1 & 1 & 1 & 0 & 0
\end{pmatrix}$&   $dP_4/\mathbb{P}(1,2)$   \\

\#39991 & 1/7   & $\begin{pmatrix}
t & s & \xi & x_1 &y_2 & y_3 & x_2 & y_1 & y_4\\
0 & 7 & 2 & 4 & 3 & 3 & 3 & 1 & 2\\
1 & 4 & 1 & 2 & 1 & 1 & 1 & 0 & 0
\end{pmatrix}$&   $dP_4/\mathbb{P}(1,2)$   \\

\#39970 & 1/5   & $\begin{pmatrix}
t & s & \xi & x_1 &y_3 & y_2 & y_1 & x_2 & y_4\\
0 & 5 & 2 & 4 & 5 & 3 & 3 & 1 & 2\\
1 & 3 & 1 & 2 & 2 & 1 & 1 & 0 & 0
\end{pmatrix}$&   $dP_3/\mathbb{P}(1,2)$   \\
        
\#39969 & 1/7   & $\begin{pmatrix}
t & s & \xi & x_1 &y_3 & x_2 & y_2 & y_1 & y_4\\
0 & 7 & 2 & 4 & 5 & 3 & 3 & 1 & 2\\
1 & 4 & 1 & 2 & 2 & 1 & 1 & 0 & 0
\end{pmatrix}$&     $dP_3/\mathbb{P}(1,2)$  \\

\#39968  & 1/11  & $\begin{pmatrix}
t & s & \xi & x_1 &y_3 & x_2 & y_2 & y_1 & y_4\\
0 & 11 & 2 & 8 & 5 & 3 & 3 & 1 & 2\\
1 & 6 & 1 & 4 & 2 & 1 & 1 & 0 & 0
\end{pmatrix}$&   $dP_3/\mathbb{P}(1,2)$   \\

\#39961 & 1/5   & $\begin{pmatrix}
t & s & \xi & x_1 &y_3 & y_2 & y_1 & x_2 & y_4\\
0 & 5 & 2 & 4 & 7 & 5 & 3 & 1 & 2\\
1 & 3 & 1 & 2 & 3 & 2 & 1 & 0 & 0
\end{pmatrix}$&   $dP_2/\mathbb{P}(1,2)$   \\

\#39607 & 1/5 (II$_2$)    & $\begin{pmatrix}
t & s_0 & \xi & y & s_1 & s_2 & v & u & z \\
1 & 3 & 1 & 2 & 3 & 3 & 2 & 1 & 1\\
3 & 4 & 1 & 2 & 3 & 2 & 1 & 0 & 0
\end{pmatrix}$&   $dP_1/\mathbb{P}^1$   \\

\#39578  & 1/7  & $\begin{pmatrix}
t & s & \xi & x_1 &y_3 & x_2 & y_2 & y_1 & y_4\\
1 & 4 & 1 & 2 & 3 & 2 & 2 & 1 & 2\\
3 & 5 & 1 & 2 & 2 & 1 & 1 & 0 & 0
\end{pmatrix}$&   $dP_2/\mathbb{P}(1,2)$    \\

\#39576  & 1/9  & $\begin{pmatrix}
t & s & \xi & x_1 &y_3 & x_2 & y_2 & y_1 & y_4\\
1 & 5 & 1 & 2 & 3 & 2 & 2 & 1 & 2\\
3 & 6 & 1 & 2 & 2 & 1 & 1 & 0 & 0
\end{pmatrix}$&   $dP_2/\mathbb{P}(1,2)$   \\
        \bottomrule
    \end{tabular}
\caption{Elementary birational links to a fibration (cases in Table \ref{tab:dimlinsys} excluded).\ The family \#39607 is embedded in the weighted projective space $\mathbb{P}^7(2,3,3,4,5,5,6,7)$ with coordinates $\xi,u,z,y,v,s_0,s_1,s_2$.}
\label{tab:fibrations}
\end{table} 

\begin{Ex} 
Let $X\subset \mathbb{P} \coloneqq \mathbb{P}(1,2,2,3,4,5,7,5)$ be a quasi-smooth member of the family \#39961, with 
$x_2, \xi, y_4, y_1, x_1, y_2, y_3, s$ the homogeneous variables of $\mathbb{P}$.\

We take the toric blowup $\Phi \colon T \rightarrow \mathbb{P}$ centred at $\mathbf{p_s}=(0: \cdots : 0 :1) \in \mathbb{P}$ and restrict it to the unique Kawamata blowup $\varphi \colon Y \rightarrow X$ centred at $\mathbf{p_s}$.\ 
The point $\mathbf{p_s}$ is a cyclic quotient singularity of type $\frac{1}{5}(1,2,3)$ and local analytical variables $\xi, x_1, x_2$ called the orbinates.\ 
By Lemma \ref{lem:lift}, in order to restrict $\Phi$ to the the Kawamata blowup of $X$ at $\mathbf{p_s}$ we need for $T$ to have a certain bi-grading: the one relative to \#39961 is listed in Table \ref{tab:fibrations}.\ In particular, 
\begin{equation*}
\Mov(T)=\langle \big(\begin{smallmatrix}
  7 \\
  4 
\end{smallmatrix}\big),\big(\begin{smallmatrix}
  1 \\
  0 
\end{smallmatrix}\big)  \rangle \subset 
\langle \big(\begin{smallmatrix}
  0 \\
  1 
\end{smallmatrix}\big),\big(\begin{smallmatrix}
  1 \\
  0 
\end{smallmatrix}\big)  \rangle = \Eff(T) \subset \mathbb{R}^2.\
\end{equation*}

Now we run the 2-ray game on $T$ following the movable cone in Case I.b of Figure \ref{fig:fibr} and then we restrict it to $Y$.\ After saturation with respect to the new variable $t$, the equations defining $Y:=\Phi^{-1}_* X \subset T$ are the following
\begin{equation*}
\begin{cases}
x_2 \xi y_1 - y_1^2 - y_4 x_1 +  x_2 y_2 =0\\
- x_2^7 t^3 -  x_2^3 \xi^2 t -  x_2 \xi^3 +  x_2^2 \xi y_1 t +  x_2 \xi x_1 -  x_2^2 y_2 t - y_1 x_1 + y_4 s =0\\
\xi^3 y_4 - y_4^4 t^3 +  x_2 \xi y_2 - y_1 y_2 +  x_2 y_3 =0\\
x_2^2 \xi^3 t +  x_2 y_4^2 y_1 t^3 +  x_2^2 \xi x_1 t +  x_2^2 y_4 x_1 t^2 - x_1^2 - y_1 s =0\\
x_2^6 y_1 t^3 -  x_2 y_4^4 t^4 + \xi^3 y_1 - y_4^3 y_1 t^3 -  x_2 y_4^2 x_1 t^2 +  x_2^2 y_3 t + x_1 y_2 =0\\
x_2^5 y_4 y_1 t^3 - y_4^5 t^4 - y_4^3 x_1 t^2 +  x_2 y_4 y_3 t + y_2^2 - y_1 y_3 =0\\
x_2^7 y_1 t^4 -  x_2^6 x_1 t^3 +  x_2^2 \xi^4 t -  x_2^2 y_4^4  t^5-  x_2 y_4^3 y_1 t^4 +  x_2^2 \xi y_4 x_1 t^2\\ -  x_2^2 y_4^2 x_1 t^3 - \xi^3 x_1 + y_4^3 x_1 t^3 -
     x_2 y_4 y_1 x_1 t^2 -  x_2 y_4^2 y_2 t^3 +  x_2^3 y_3 t^2 + y_2 s =0\\
 x_2^5 y_4 x_1 t^3 -  x_2^6 y_2 t^3 -  x_2 \xi y_4^4 t^4 + y_4^4 y_1 t^4 -  x_2 \xi y_4^2 x_1 t^2 \\+ y_4^2 y_1 x_1 t^2 - \xi^3 y_2 + y_4^3 y_2 t^3 +  x_2^2 \xi y_3 t -  x_2 y_1 y_3 t - x_1 y_3 =0\\
- x_2^6 \xi^3 t^3 +  x_2^6 y_4^3 t^6 -  x_2^6 \xi x_1 t^3 +  x_2^2 \xi^2 y_4^3 t^4 +  x_2^5 y_1 x_1 t^3 - \xi^6 + \xi^3 y_4^3 t^3 -  x_2 \xi y_4^3 y_1 t^4 \\+
x_2^2 \xi^2 y_4 x_1 t^2 - y_4^4 x_1 t^4 -  x_2 \xi y_4 y_1 x_1 t^2 -  x_2 \xi y_4^2 y_2 t^3 +  x_2 y_4^3 y_2 t^4 \\- y_4^2 x_1^2 t^2 + y_4^2 y_1 y_2 t^3 +  x_2 y_4 x_1 y_2 t^2 +
     x_2 x_1 y_3 t - y_3 s =0 
\end{cases}
\end{equation*}

These consist of the five maximal Pfaffians $\Pf_i=0$ together with the four unprojection equations $sy_j-g_j=0$.\ Also, $Y$ is inside the rank 2 toric variety $T$ whose irrelavant ideal is $(t,s) \cap (\xi,x_1,y_3,y_2,y_1,x_2,y_4)$.\ 
By Lemma \ref{lem:iso}, the map $\tau_0 \colon T \dashrightarrow T_1$ restricts to an isomorphism on $Y$ and therefore we can assume $Y\subset T_1$ where $T_1$ has irrelevant ideal $I_1=(t,s,\xi,x_1) \cap (y_3,y_2,y_1,x_2,y_4)$.  
Crossing the $y_3$-wall contracts the locus $\mathbb{V}(y_2,y_1,x_2,y_4) \subset T_1$.\ Its restriction to $Y_1$ is $\mathbb{V}(y_3,y_2,y_1,x_2,y_4)\subset \mathbb{V}(I_1)$, where $I_1$ is the ideal defining $Y_1 \subset T_1$.\ Hence, the contraction happens away from $Y_1$, so $\tau_1$ restricts to an isomorphism $Y_1 \cong Y_2$.\ 
Just as before we just set $Y_1 \cong Y_2 \subset T_2$ where $I_2=(t,s,\xi,x_1,y_3) \cap (y_2,y_1,x_2,y_4)$.\ Crossing the $y_2$-wall restricts to a contraction of $C_2 \coloneqq (y_1=x_2=y_4=0)|_{Y_2} \simeq \mathbb{P}(7,1)$
and an extraction of $C_3 \coloneqq (t=s=\xi=x_1=0)|_{Y_3} \simeq \mathbb{P}(1,5)$.\ 
Hence, the map $\tau_2$ corresponds to a toric flip over a point, denoted by $(-7,-1,1,5)$ and $Y_3 \subset T_3$ with irrelevant ideal $I_3=(t,s,\xi,x_1,y_3,y_2) \cap (y_1,x_2,y_4)$.\ The map $\tau_3$ restricts to 
\begin{equation*}
(x_2=y_4=0)|_{Y_3}= (y_1=x_2=y_4=0) \subset \mathbb{V}(I_3) 
\end{equation*}
where $I_3$ is the ideal defining $Y_3 \subset T_3$.\ Therefore the small contraction $\tau_3 \colon T_3 \dashrightarrow T_4$ happens away from $Y_3$.\ Finally, we have a map $\varphi' \colon Y_4 \rightarrow \mathcal{F}' = \Proj \mathbb{C}[x_2,y_4]$.\
A generic fibre is a surface $S$ given by 
\begin{equation*}
t\xi^3 + t^2y_2 + t\xi y_2 - y_2^2 + t^2\xi y_1 - \xi^3y_1 - ty_2y_1 - \xi y_1y_2 - t^2y_1^2       + y_2y_1^2 =0 
\end{equation*}
inside $\mathbb{P}(1_t,1_{\xi},1_{y_1},2_{y_2})$.\ Hence, $S$ is a del Pezzo surface of degree 2.\ Therefore we have the diagram
\begin{equation*}
\begin{tikzcd} 
& Y \arrow[swap]{dl}{\varphi} \arrow[]{r}{\simeq} & Y_1 \arrow[]{r}{\simeq} &  Y_2\arrow[swap]{dr}{\alpha_2} \arrow[dashed]{rr}{(-7,-1,1,5)} & & Y_3 \arrow[]{dl}{\beta_2} \arrow[]{r}{\simeq}   & Y_4\arrow{dr}{\varphi'}  & \\
 X &  &  &  & \mathcal{F}_2 &  &  & \mathbb{P}(1_{x_2},2_{y_4})  
\end{tikzcd}
\end{equation*}
where $\varphi' \colon Y_4 \rightarrow \mathbb{P}(1_{x_2},2_{y_4})$ is a del Pezzo fibration of degree two.\
\end{Ex}

\subsection{Case II: Divisorial contractions}

\begin{figure*}[h!]
        \centering
        \begin{subfigure}[b]{0.3\textwidth}
            \centering
            \resizebox{\linewidth}{!}{
\begin{tikzpicture}[every node/.style={scale=0.333}]
  \coordinate (A) at (0, 0);
  \coordinate [label={left:$t$}] (E) at (0, 1);
  \coordinate [label={above:$\xi,\,x_1$}] (K) at (1, 1);

\coordinate [label={left:$s$}] (5) at (0.5,1);
\coordinate [label={right:$z_1$}] (A1) at (1.2,0.8);
\coordinate [label={above:$z_2$}] (A2) at (1.4,0.6);
\coordinate [label={right:$z_3$}] (A3) at (1.4,0.4);
\coordinate [label={right:$z_4$}] (A4) at (1.4,0.2);
\coordinate [label={right:$z_5$}] (M2) at (1.3,0);

  \draw  (A) -- (E);
  \draw  (A) -- (K);
	\draw [very thick,color=red] (A) -- (5);
	\draw (A) -- (A1);
	\draw (A) -- (A2);
	\draw (A) -- (A3);
	\draw [very thick,color=red] (A) -- (A4);
    \draw  (A) -- (M2);

\end{tikzpicture}
}
            \caption[Case II.a.1]
            {{\small Case II.a.1}}    
            \label{fig:IIa1}
        \end{subfigure}
        \hfill
        \begin{subfigure}[b]{0.3\textwidth}
            \centering 
\resizebox{\linewidth}{!}{
\begin{tikzpicture}[every node/.style={scale=0.333}]
  \coordinate (A) at (0, 0);
  \coordinate [label={left:$t$}] (E) at (0, 1);
  \coordinate [label={above:$\xi,x_1$}] (K) at (1, 1);

\coordinate [label={left:$s$}] (5) at (0.5,1);
\coordinate [label={right:$z_1$}] (A1) at (1.2,0.8);
\coordinate [label={right:$z_2,\,z_3$}] (A2) at (1.4,0.6);
\coordinate [label={right:$z_4$}] (A3) at (1.4,0.3);

\coordinate [label={right:$z_5$}] (M2) at (1.3,0);

  \draw (A) -- (E);
  \draw  (A) -- (K);
	\draw [very thick,color=red](A) -- (5);
	\draw (A) -- (A1);
	\draw (A) -- (A2);
	\draw [very thick,color=red] (A) -- (A3);
    \draw (A) -- (M2);

\end{tikzpicture}
    }    
            \caption[]
            {{\small Case II.a.2}}    
            \label{fig:IIa2}
        \end{subfigure}
        \begin{subfigure}[b]{0.3\textwidth}   
            \centering 
            \resizebox{\linewidth}{!}{
\begin{tikzpicture}[every node/.style={scale=0.333}]
  \coordinate (A) at (0, 0);
  \coordinate [label={left:$t$}] (E) at (0, 1);
  \coordinate [label={above:$\xi,\,x_1$}] (K) at (1, 1);

\coordinate [label={left:$s$}] (5) at (0.5,1);
\coordinate [label={right:$z_1,\,z_2$}] (A1) at (1.2,0.8);
\coordinate [label={right:$z_3$}] (A2) at (1.4,0.6);
\coordinate [label={right:$z_4$}] (A3) at (1.4,0.3);
\coordinate [label={right:$z_5$}] (M2) at (1.3,0);

  \draw (A) -- (E);
  \draw (A) -- (K);
	\draw [very thick,color=red](A) -- (5);
	\draw (A) -- (A1);
	\draw (A) -- (A2);
	\draw [very thick,color=red] (A) -- (A3);
    \draw (A) -- (M2);
\end{tikzpicture}
}
            \caption[]
            {{\small Case II.a.3}}    
            \label{fig:IIa3}
        \end{subfigure}
  
        \vskip\baselineskip
        \begin{subfigure}[b]{0.4\textwidth}   
            \centering 
            \resizebox{\linewidth}{!}{
\begin{tikzpicture}[every node/.style={scale=0.333}]
  \coordinate (A) at (0, 0);
  \coordinate [label={left:$t$}] (E) at (0, 1);
  \coordinate [label={above:$\xi,\,x_1$}] (K) at (1, 1);

\coordinate [label={left:$s$}] (5) at (0.5,1);
\coordinate [label={right:$z_1$}] (A1) at (1.2,0.8);
\coordinate [label={right:$z_2$}] (A2) at (1.4,0.6);
\coordinate [label={right:$z_3,\,z_4\, \phantom{.......}$}] (A3) at (1.4,0.3);
\coordinate [label={right:$z_5$}] (M2) at (1.3,0);

  \draw  (A) -- (E);
  \draw  (A) -- (K);
	\draw  [very thick,color=red] (A) -- (5);
 	\draw (A) -- (A1);
	\draw (A) -- (A2);
	\draw  [very thick,color=red](A) -- (A3);

    \draw   (A) -- (M2);

\end{tikzpicture}
}
            \caption[]
            {{\small Case II.b.1}}    
            \label{fig:IIb1}
        \end{subfigure}
                \begin{subfigure}[b]{0.4\textwidth}   
            \centering 
            \resizebox{\linewidth}{!}{
\begin{tikzpicture}[every node/.style={scale=0.333}]
  \coordinate (A) at (0, 0);
  \coordinate [label={left:$t$}] (E) at (0, 1);
  \coordinate [label={above:$\xi,\,x_1$}] (K) at (1, 1);

\coordinate [label={left:$s$}] (5) at (0.5,1);

\coordinate [label={right:$z_1,\,z_2$}] (A2) at (1.4,0.7);
\coordinate [label={right:$z_3,\,z_4\, \phantom{.......}$}] (A3) at (1.4,0.3);

\coordinate [label={right:$z_5$}] (M2) at (1.3,0);

  \draw  (A) -- (E);
  \draw  (A) -- (K);
	\draw [very thick,color=red] (A) -- (5);

	\draw (A) -- (A2);
	\draw [very thick,color=red] (A) -- (A3);

    \draw (A) -- (M2);

\end{tikzpicture}
}
            \caption[]
            {{\small Case II.b.2}}    
            \label{fig:IIb2}
        \end{subfigure}
        \caption{\small The possible effective cones of $T$ ending with a contraction of the divisor $D_{z_5} \colon (z_5=0)$. In cases \hyperref[fig:IIa1]{II.a.1}, \hyperref[fig:IIa2]{II.a.2}, \hyperref[fig:IIa3]{II.a.3} the divisor $D_{z_5}$ is contracted to the point $\mathcal{F}'=\Proj \mathbb{C}[z_4]$ and in cases \ref{fig:IIb1} and \ref{fig:IIb2} to the rational curve $\mathcal{F}'=\Proj \mathbb{C}[z_3,z_4]$.\ The variables $z_1,\ldots,z_5$ are $y_1,\ldots,y_4,x_2$ up to permutation.\ The exceptional divisor of the Kawamata blowup $\varphi \colon Y \rightarrow X$ is $E\colon (t=0)$.}
        \label{fig:pliconj}
    \end{figure*}
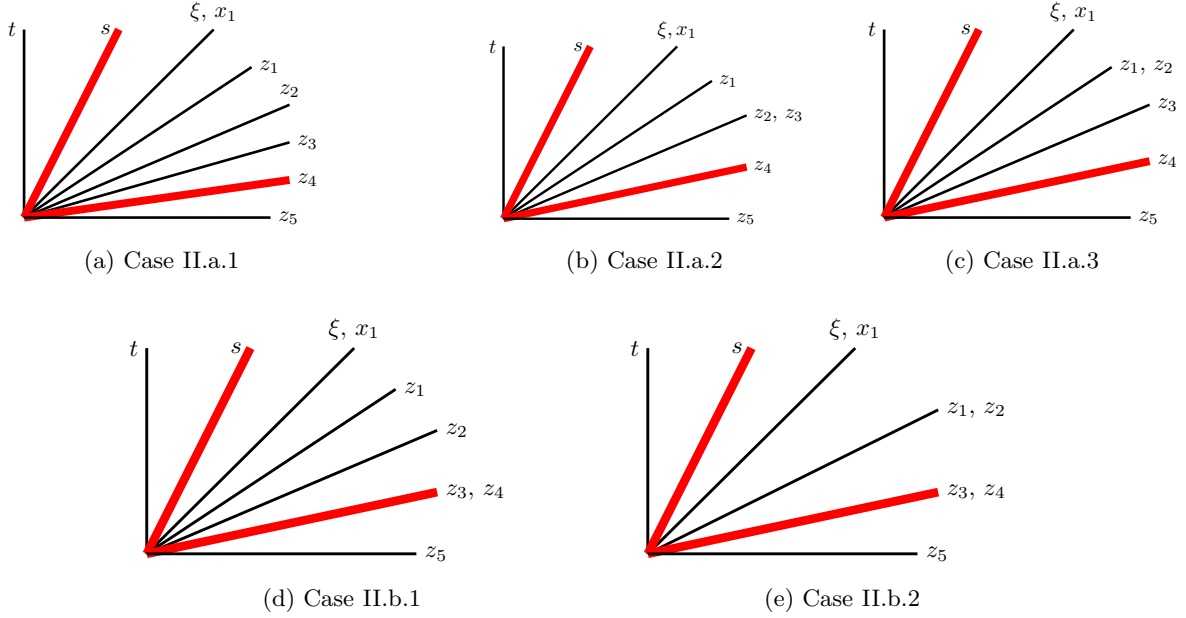

In each case in Figure \ref{fig:pliconj}, whose notation we refer to, the movable cone of $T$ is strictly contained in the effective cone of $T$.\ Hence we have a diagram of toric varieties
\begin{equation*}
\begin{tikzcd} 
T\arrow[swap]{dd}{\Phi} \arrow[dashed]{rr}{\tau}& &  T'\arrow{dd}{\Phi'}\\
&&\\
 \mathbb{P} \arrow[swap,dashed]{rr}{\sigma} &  &\mathcal{F}'  
\end{tikzcd}
\end{equation*}
where $\Phi$ and $\Phi'$ are divisorial contractions and $\tau$ is a small $\mathbb{Q}$-factorial modification.\ 
As usual, we restrict $\Phi \colon T \rightarrow \mathbb{P}$ to be the unique Kawamata blowup $\varphi \colon E \subset Y \rightarrow \mathbf{p_s} \in X$.\ 
By Lemmas \ref{lem:iso} and \ref{lem:wall2}, the map $\tau|_Y \colon Y \dashrightarrow Y' $ is an isomorphism followed by a finite sequence of isomorphisms or flips.\ 
Referring to the notation in Figure \ref{fig:pliconj}, by assumption the rays $\mathbb{R}_+[z_4]$ and $\mathbb{R}_+[z_5]$ are given by linearly independent vectors $v_4$ and $v_5$ in $\mathbb{Z}^2$.\ Let $-d \not= 0$ be the determinant of the matrix $B$ whose columns are $v_4$ and $v_5$; without loss of generality we can assume that $d >0$ (cf \cite[Lemma 2.4]{hamidplia}).\ Then, there is $A \in \GL(2,\mathbb{Z})$ such that 
\begin{equation*}
A\cdot B = \begin{pmatrix}
0 & -d \\
d & 0 
\end{pmatrix} \; .
\end{equation*}
After a row operation on the grading of $T'$ via the matrix $A$, $T'$ is isomorphic to a toric variety with $\mathbb{C}^* \times \mathbb{C}^*$-action given by
\begin{equation} \label{mat}
\begin{array}{ccccccccccccc}
             &       & t  & s &   \xi & x_1 & z_1 & z_2 & z_3 & z_4 & z_5 \\
\actL{T'}   &  \lBr &  \kappa_0 & \kappa_1 &   \kappa_2 & \kappa_3 & \kappa_4 & \kappa_5 & \kappa_6 & 0 & -d &   \actR{}\\
             &       & \lambda_0 & \lambda_1 &  \lambda_2 & \lambda_3 & \lambda_4 & \lambda_5 & \lambda_6 & d & 0 &  
\end{array}
\end{equation}
where every $2 \times 2$ minor is non-positive.\ In cases II.a the irrelevant ideal of $T'$ is $(t , s,  \xi, x_1, z_1, z_2, z_3) \cap (z_4, z_5)$, and $(t , s,   \xi, x_1, z_1, z_2) \cap (z_3, z_4, z_5)$ in cases II.b.\

\paragraph{Cases II.a.\ Divisorial contractions to a point.\ } We have $\kappa_6 \not = 0$ by assumption.\ The map $\Phi'$ is given by the sections multiples of $D_{z_4}$ (in the notation of \cite[Section 4.1.7]{2raybrown}) and is a divisorial contraction to a point in $\mathcal{F}'$.\ This can be realised as 
\begin{align*}
\Phi' \colon T' &\longrightarrow \mathcal{F}' \\
(t , s,  \xi, x_1, z_1, z_2, z_3, z_4, z_5) & \longmapsto (tz_5^{\frac{\kappa_0}{d}},sz_5^{\frac{\kappa_1}{d}},\xi z_5^{\frac{\kappa_2}{d}},x_1 z_5^{\frac{\kappa_3}{d}}, z_1z_5^{\frac{\kappa_4}{d}}, z_2 z_5^{\frac{\kappa_5}{d}}, z_3z_5^{\frac{\kappa_6}{d}},z_4) \; .
\end{align*}
Assume $X$ is not the family \#39660.\ Then, $\mathcal{F}'= \mathbb{P}(\lambda_0, \ldots, \lambda_6, d)$ and $\Phi'$ contracts the divisor $E' \colon (z_5=0) \simeq \mathbb{P}(\kappa_0, \ldots,\kappa_6)$ to the point $\mathbf{p}=(0:\cdots : 0: 1) \in \mathcal{F}'$.\ In particular $\mathbf{p}$ is smooth whenever $d=1$.\ The contraction $\Phi'$ restricts to a weighted blowup $\varphi' \colon Y' \rightarrow X'$ as in Table \ref{tab:11111}.\

\begin{table}[ht!]
    \centering
   \resizebox{\textwidth}{!}{
   \begin{tabular}[t]{l c c c c}
        \toprule
        ID   & Centre     & $T'$ &   $\varphi'$           & $X' \subset \mathcal{F}'$     \\ \midrule

\#39569  & 1/7 (II$_2$) & $\begin{pmatrix}
t & s_0 & \xi & y & s_1 & s_2 & v & u & z \\
5 & 6 & 1 & 3 & 4 & 2 & 1 & 0 & -1\\
3 & 5 & 1 & 3 & 4 & 3 & 2 & 1 & 0
\end{pmatrix}$& $(5,1,3,1)$ & $ X'_{6,6} \subset \mathbb{P}(1^2,2,3^3)$ \\

\#39660  & 1/17 & $\begin{pmatrix}
t & s & \xi & x_1 & y_3 & x_2 & y_2 & y_1 & y_4 \\
3 & 10 & 1 & 6 & 2 & 1 & 1 & 0 & -2\\
1 & 9 & 1 & 6 & 3 & 2 & 2 & 1 & 0
\end{pmatrix}$& $\frac{1}{2}(1,1,1)$ & $ X'_{4,5} \subset \mathbb{P}(1^3,2^2,3)/ \bm{\mu}_2$ \\

\#39890 & 1/11  & $\begin{pmatrix}
t & s & \xi & x_1 & y_3 & y_2 & y_1 & y_4 & x_2\\
8 & 15 & 2 & 10 & 5 & 3 & 1 & 0 & -3\\
1 & 6 & 1 & 5 & 4 & 3 & 2 & 3 & 0
\end{pmatrix}$& $ \frac{1}{3}(2,10,5,3,1)$ & 
$ X' \subset \mathbb{P}(1^2,2,3^2,4,5,6)$   \\
        
\#39906 & 1/7  & $\begin{pmatrix}
t & s & \xi & x_1 & y_3 & y_2 & y_1 & y_4 & x_2\\
4 & 9 & 2 & 6 & 5 & 3 & 1 & 0 & -1\\
1 & 4 & 1 & 3 & 3 & 2 & 1 & 1 & 0
\end{pmatrix}$& $(2,5,3,1)$ & $ X' \subset \mathbb{P}(1^4,2,3,4)$ \\

\#39912 & 1/11 & $\begin{pmatrix}
t & s & \xi & x_1 & y_3 & x_2 & y_2 & y_4 & y_1\\
4 & 13 & 2 & 6 & 3 & 3 & 1 & 0 & -1\\
1 & 6 & 1 & 3 & 2 & 2 & 1 & 1 & 0
\end{pmatrix}$& $(2,3,3,1)$ & $ X' \subset \mathbb{P}(1^4,2^2,3)$ \\

\#39913 & 1/5 & $\begin{pmatrix}
t & s & \xi & x_1 & y_3 & y_2 & y_1 & y_4 & x_2\\
4 & 9 & 2 & 6 & 3 & 3 & 1 & 0 & -1\\
1 & 4 & 1 & 3 & 2 & 2 & 1 & 1 & 0
\end{pmatrix}$& $(2,3,3,1)$ & $ X' \subset \mathbb{P}(1^3,2^2,3,4)$ \\

\#39928 & 1/13 & $\begin{pmatrix}
t & s & \xi & x_1 & x_2 & y_3 & y_2 & y_4 & y_1\\
4 & 15 & 2 & 4 & 7 & 3 & 1 & 0 & -1\\
1 & 7 & 1 & 2 & 4 & 2 & 1 & 1 & 0
\end{pmatrix}$ & $(2,4,7,3,1)$ & $ X'_{3,4} \subset \mathbb{P}(1^4,2^2)$   \\

\#39929  & 1/5  & $\begin{pmatrix}
t & s & \xi & x_1 & y_3 & y_2 & y_1 & y_4 & x_2\\
4 & 7 & 2 & 4 & 7 & 3 & 1 & 0 & -1\\
1 & 3 & 1 & 2 & 4 & 2 & 1 & 1 & 0
\end{pmatrix}$& $(4,2,7,3,1)$ & $ X' \subset \mathbb{P}(1^4,2^2,3)$ \\

\#39934   & 1/5  & $\begin{pmatrix}
t & s & \xi & x_1 & y_3 & y_2 & y_1 & y_4 & x_2\\
4 & 7 & 2 & 4 & 3 & 3 & 1 & 0 & -1\\
1 & 3 & 1 & 2 & 2 & 2 & 1 & 1 & 0
\end{pmatrix}$& $(2,3,3,1)$ & $ X' \subset \mathbb{P}(1^4,2^3,3)$ \\
        \bottomrule
    \end{tabular}}
\caption{We list the restriction $\Phi'|_{Y'}=\varphi'$ and the model to which $\varphi'$ contracts to.\ In each case $\varphi'$ is a weighted blowup with weights $\frac{1}{r}(a_1,\ldots,a_l)$ with $r\geq 1$.\ For case \#39890 $\varphi'$ is a contraction to a hyperquotient singularity; in the other instances, $\varphi'$ is a contraction to a Gorenstein point.\ The family \#39569 is embedded in the weighted projective space $\mathbb{P}^7(2,3,5,6,7,7,8,9)$ with coordinates $\xi,z,u,y,v,s_0,s_1,s_2$.\ }
\label{tab:11111}
\end{table}

The following explicit example illustrates a link that terminates with a 3-fold in a fake weighted projective space.\
\begin{Ex}

Consider the deformation family with ID \#39660 in Tom format.\ This is $X \subset \mathbb{P} \coloneqq \mathbb{P}(2,2,3,5,5,7,12,17)$ with homogeneous coordinates $\xi,\,y_4,\,y_1,\,x_2,\,y_2,\,y_3,\,x_1,\,s$.\ By Lemma \ref{lem:lift} we know that the weighted blowup of $\mathbf{p_s}=(0:\ldots : 0 : 1) \in \mathbb{P}$ restricts to the Kawamata blowup $\varphi \colon Y \subset T \rightarrow X \subset \mathbb{P}$ provided that the bi-grading of $T$ is
\begin{equation*}
\begin{array}{cccc|ccccccccc}
             &       & t  & s &   \xi & x_1 & y_3 & x_2 & y_2 & y_1 & y_4 & \\
\actL{T}   &  \lBr &  3 & 10 &   1 & 6 & 2 & 1 & 1 & 0 & -2 &   \actR{}\\
             &       & 1 & 9 &  1 & 6 & 3 & 2 & 2 & 1 & 0 &  
\end{array}
\end{equation*}
up to multiplication by a matrix in $\GL(2,\mathbb{Z})$ (see \cite[Lemma 2.4 and Example 2.13]{hamidplia}).\ We have a sequence of small $\mathbb{Q}$-factorial modifications  $\tau \colon T \dashrightarrow T'$ where $T'$ has the same Cox ring as $T$ but whose irrelevant ideal is $(t,s,\xi,x_1,y_3,x_2,y_2) \cap (y_1,y_4)$.\ 
Let $T^{y_4}$ be the same rank 2 toric variety as $T'$, except that $y_4$ has bi-degree $(-1,0)$, instead of $(-2,0)$.\ Then, there is a map $q \colon T^{y_4} \rightarrow T'$ given by $y_4 \mapsto y_4^2$ and $T^{y_4}$ can be seen as a double cover of $T'$.\ 
On the other hand, $T'=\Proj \mathbb{C}[t,s,\xi,x_1,y_3,x_2,y_2,y_1,y_4]^{\bm{\mu}_2}$, where $\bm{\mu}_2$ is the multiplicative cyclic group of order 2 acting on $\mathbb{C}[t,s,\xi,x_1,y_3,x_2,y_2,y_1,y_4]$ via
\begin{equation*}
 \epsilon \cdot (t,s,\xi,x_1,y_3,x_2,y_2,y_1,y_4) = (t,s,\xi,x_1,y_3,x_2,y_2,y_1,\epsilon y_4) \; .
\end{equation*}
Hence, $q$ is the quotient map of $T^{y_4}$ under this group action.\ Consider the map 
\begin{align*}
\Phi^{y_4} \colon T^{y_4} &\longrightarrow \mathbb{P}(1,9,1,6,3,2,2,1) \\
(t , s,   \xi, x_1, y_3, x_2, y_2, y_1, y_4) & \longmapsto (t y_4^3,s y_4^{10},\xi y_4, x_1 y_4^6, y_3 y_4^2, x_2 y_4, y_2 y_4, y_1) \; .
\end{align*}
This is the map given by the sections multiples of $H^0(T^{y_4},\mathcal{O}_{T^{y_4}}(0,1))$ and it corresponds to a divisorial contraction to a point in $\mathbb{P}(1,9,1,6,3,2,2,1)$.\ 
Since $\Phi^{y_4}$ is an isomorphism away from the locus $(y_4=0)$, the action on $T^{y_4}$ is carried through to $\mathbb{P}(1,9,1,6,3,2,2,1)$.\ 

The action at each point of this weighted projective space is then given by 
\begin{align*}
 \epsilon \cdot (t: s: \xi: x_1: y_3: x_2: y_2: y_1: y_4) &= (\epsilon t: s: \epsilon \xi : x_1 : y_3 : \epsilon x_2 : \epsilon y_2 : y_1) \\
 &=( t:\epsilon   s: \xi : x_1 :\epsilon y_3 : \epsilon x_2 : \epsilon y_2 : \epsilon y_1) \; .
\end{align*}

Then, $\Phi^{y_4}$ descends to a divisorial contraction of quotient spaces 
\begin{equation*}
\Phi' \colon T'=T^{y_4}/\bm{\mu}_2 \rightarrow \mathbb{P}(1,9,1,6,3,2,2,1)/\bm{\mu}_2 \; .
\end{equation*} 
Let $Y'$ be the image of $\tau$ restricted to $Y$. The map $\Phi'$ restricts to a divisorial contraction $\varphi' \colon Y' \rightarrow X'$ to the point $\mathbf{p_{y_1}} \sim \frac{1}{2}(1,1,1) \in X'$ in the complete intersection of a quartic and a quintic $X' \subset \mathbb{P}(1,1,3,2,2,1)/\bm{\mu}_2$ with homogeneous variables $t,\xi, y_3, x_2, y_2, y_1$ which is given by the equations
\begin{equation*}
\begin{cases}
t^4 + t \xi^3 - y_3 y_1 - y_2^2 = 0 \\
t x_2^2 + t y_1^4 - \xi^5 - \xi^2 x_2 y_1 + y_3 y_2 = 0  \; .
\end{cases}
\end{equation*}
Note that $X'$ is invariant under the action of $\bm{\mu}_2$.\ By Lemma \ref{lem:discr}, the map $\varphi'$ has discrepancy $\frac{1}{2}$ and its exceptional divisor is isomorphic to $\mathbb{P}^2$.\ Hence, $\varphi'$ is the Kawamata blowup centred at $\mathbf{p_{y_1}}\in X'$.\
\end{Ex}

\paragraph{Cases II.b.\ Divisorial contractions to a rational curve.\ } By assumption, $\kappa_6 = 0$ in the grading of $T'$ in \eqref{mat}.\ The map $\Phi'$ is given by the sections that are multiples of the divisor $(z_4 =0)$, and is a divisorial contraction to the rational curve $\Gamma' := \Proj \mathbb{C}[\lambda_6,d]$.\ By Lemma \ref{lem:eqnsY}, $\Phi'$ restricts to a divisorial contraction to a curve at the level of 3-folds.\ By terminality, it follows that $\gcd(\lambda_6,d)=1$ since, otherwise, the curve $\Gamma'$ would be a line of singularities.\ In fact, by looking at each case we can see that $d=1$.\ The map $\Phi'$ is then 
\begin{align*}
\Phi' \colon T' &\longrightarrow \mathcal{F}' \\
(t , s,   \xi, x_1, z_1, z_2, z_3, z_4, z_5) & \longmapsto (tz_5^{\kappa_0},sz_5^{\kappa_1},\xi z_5^{\kappa_2},x_1 z_5^{\kappa_3}, z_1z_5^{\kappa_4}, z_2 z_5^{\kappa_5}, z_3,z_4) \; .
\end{align*}
Hence, the divisor $(z_5=0) \subset T'$ which is isomorphic to $\mathbb{P}(\kappa_0,\ldots,\kappa_5)$ is contracted to $\Gamma' \subset \mathbb{P}(\lambda_0,\ldots,\lambda_6,1)$.\ In Table \ref{tab:1121} we summarise the details regarding all cases falling into II.b.\

\begin{table}[ht!]
    \centering
   \resizebox{\textwidth}{!}{\begin{tabular}[t]{l c c c c}
        \toprule
        ID    & Centre    & $T'$              & $\Gamma' \subset  X' \subset \mathcal{F}'$     \\ \midrule

\#39557  & 1/11 &  $\begin{pmatrix}
t&s&\xi& x_1 &y_3 & y_2 & y_1 & x_2 & y_4\\
5& 8 & 1 & 3 & 2 & 1 & 0 & 0 & -1\\
8&15&2&6&5 & 3 & 1 & 1 & 0
\end{pmatrix}$ &   $\mathbb{P}^1\subset X'_{10} \subset \mathbb{P}(1,1,2,3,5)$   \\
        
\#39605 & 1/13 & $\begin{pmatrix}
t&s&\xi& x_1 &y_3 & y_2 & y_1 & x_2 & y_4\\
3& 8 & 1 & 5 & 2 & 1 & 0 & 0 & -1\\
4&15&2&10&5 & 3 & 1 & 1 & 0
\end{pmatrix}$&    $\mathbb{P}^1\subset X'_{6,8} \subset \mathbb{P}(1,1,2,3,4,5)$ \\
        
\#39675 & 1/9 & $\begin{pmatrix}
t&s&\xi& x_1 &x_2 & y_3 & y_1 & y_2 & y_4\\
3 & 6  & 1 & 1 & 2 & 1 & 0 & 0 & -1\\
4 & 11 & 2& 2 &5 & 3 & 1 & 1 & 0
\end{pmatrix}$& $\mathbb{P}^1 \subset X'_{6,6} \subset \mathbb{P}(1,1,2,2,3,5)$   \\

\#39678 & 1/5 & $\begin{pmatrix}
t & s & \xi & x_1 &y_2 & y_3 & x_2 & y_1 & y_4\\
3 & 4  & 1 & 1 & 1 & 1 & 0 & 0 & -1\\
4 & 7 & 2& 2 &3 & 3 & 1 & 1 & 0
\end{pmatrix}$& $\mathbb{P}^1 \subset X'_{4,6} \subset \mathbb{P}(1,1,2,2,3,3)$   \\

\#39676  & 1/7 & $\begin{pmatrix}
t & s & \xi & x_1 &x_2 & y_3 & y_2 & y_1 & y_4\\
3 & 5  & 1 & 1 & 1 & 1 & 0 & 0 & -1\\
4 & 9 & 2& 2 &3 & 3 & 1 & 1 & 0
\end{pmatrix}$& $\mathbb{P}^1 \subset X'_{4,6} \subset \mathbb{P}(1,1,2,2,3,3)$   \\

\#39898 & 1/9  &$\begin{pmatrix}
t&s&\xi& x_1 &y_3 & y_2 & y_1 & y_4 & x_2\\
3& 6 & 1 & 4 & 2 & 1 & 0 & 0 & -1\\
1 & 5 &1 &4 & 3 & 2 & 1 & 2 & 0
\end{pmatrix}$& $\mathbb{P}^1(1,2)\subset X' \subset \mathbb{P}(1^3,2^2,3,4,5)$ \\
        \bottomrule
    \end{tabular}}
\caption{We list the restriction $\Phi'|_{Y'}=\varphi'$ and the model to which $\varphi'$ contracts to.\ In each case $\varphi'$ is a contraction to a curve $\Gamma'$ inside a Fano 3-fold $X'$.\ }
\label{tab:1121}
\end{table}

\paragraph{Divisorial contractions to curves with non-rational components.\ } \label{subsect: comparison with Ducat}

Here we look into more detail at the case in which the birational link terminates with a divisorial contraction to a non-complete intersection curve.\ 
The families falling into this description are \#40663, \#40671, \#40672 and \#40993.\ Notice that the families \#40663 and \#40993 have been treated in \cite[Table~1]{ducat} and are referred to as A.3 and A.2 respectively.\ 
In that paper, Ducat constructs these two families via simple Sarkisov links initiated by blowing up a curve $\Gamma$ on a rational 3-fold.\ 

It turns out that these are not the only codimension 4 and index 2 Fano 3-folds which can be obtained in this way.\ Here we rely on the construction of $X$ as in Section \ref{sect:cons} and show that \#40671 and \#40672 are rational via a reversed procedure to the one in \cite{ducat}.\ These two examples are interesting also because the Sarkisov link at the toric level ends with a fibration while its restriction to the 3-folds ends with a divisorial contraction to a non-rational curve.\ 
Moreover, we compute the Picard rank of \#40671 and \#40672.\

\paragraph{\#40672.\ } 

Let $X \subset \mathbb{P} \coloneqq \mathbb{P}(1,1,1,2,2,2,3,3)$.\ with homogeneous coordinates $y_1,y_2,x_2,\xi,y_4,x_1,y_3,s$ be a quasi-smooth member of the family \#40672 as in Section \ref{sect:cons} and \cite{CampoC4I2}.\
    
After the Kawamata blowup of the point $\mathbf{p_s} \sim \frac{1}{3}(1,1,2)$ and the isomorphism $\tau_0$ in the birational link, we have,
\begin{equation*}
\begin{array}{cccccc|ccccccc}
             &       & t  & s &   \xi & x_1 & y_3 & y_1 & y_2 & x_2 & y_4 & \\
\actL{T_1}   &  \lBr &  -3 & -3 &   -1 & -1 & 0 & 1 & 1 & 1 & 2 &   \actR{}\\
             &       & 1 & 2 &  1 & 1 & 1 & 0 & 0 & 0 & 0 &  
\end{array} \; .
\end{equation*}
The small $\mathbb{Q}$-factorial modification $\tau_1$ contracts $\mathbb{P}(3,3,1,1) \subset T_1$ to a point and extracts $\mathbb{P}(1,1,1,2) \subset T_2$ where $T_2$ has the same Cox ring as $T_1$ but irrelevant ideal $(t,s,\xi,x_1,y_3)\cap (y_1,y_2,x_2,y_4)$.\ This restricts to the flip
\begin{equation*}
    \begin{tikzcd}[ampersand replacement=\&, column sep = 2em]
             C_1 \subset Y_1   \ar[rr, dashed ] \ar[dr, swap, "\displaystyle{\alpha_1}" ] \& {} \& Y_2 \supset C_2\ar[ld, "\displaystyle{\beta_1}" ] \\
             {} \& \mathbf{p_{y_3}} \& {}
    \end{tikzcd}
\end{equation*}
where $C_1 \colon (t+s+2 x_1^3=\xi-x_1=0) \subset \mathbb{P}(3,3,1,1)$, that is, $C_1 \simeq \mathbb{P}(3_t,1_{x_1})$ and $C_2 \colon (y_1+x_2=y_2+x_2=0) \subset \mathbb{P}(1,1,1,2)$, that is $C_2 \simeq \mathbb{P}(1_{x_2},2_{y_4})$.\ Also $-K_{Y_1} \sim (\xi =0) \subset Y_1$ and $E \sim (t=0) \subset Y_1$.\

It is clear that $-K_{Y_1}\cdot C_1=\frac{1}{3}$.\ 
On the other hand, we have $-K_{Y_2}\sim D_{y_3} - D_{x_2}$.\ Hence, $-K_{Y_2} \cdot C_2 = -\frac{1}{2}$.\ Consider the map $\Phi'$ 
\begin{align*}
    \Phi' \colon T_2 &\rightarrow \mathbb{P}(1,1,1,2) \\
    (t,s,\xi,x_1,y_3,y_1,y_2,x_2,y_4) &\mapsto (y_1,y_2,x_2,y_4) \; .
\end{align*}

Then $\Phi'$ is a fibration whose fibres are isomorphic to $\mathbb{P}(1,2,1,1,1)$.\ 
Consider the two consecutive projections 
$X \dashrightarrow X' \dashrightarrow X''$ where $X \dashrightarrow X'$ is the projection away from $\mathbf{p}_s \in X$ and $X' \dashrightarrow X''$ is the projection away from $\mathbf{p}_{y_3} \in X'$.\ The equations of $X''$ are given explicitly by 
\begin{equation*}
\setlength\arraycolsep{7pt}
\begin{pmatrix}
-y_1 & y_2 & y_2^2x_2+y_1x_2^2-x_2y_4\\
y_4-y_1^2-y_1y_2-y_1x_2-y_2x_2 & -y_4 & -y_1^4+y_1^3y_2+y_2^2y_4+y_1x_2y_4-y_4^2
\end{pmatrix}
\begin{pmatrix}
x_1 \\
\xi \\
1
\end{pmatrix}=
\begin{pmatrix}
0 \\
0 
\end{pmatrix} \; .
\end{equation*}
Let $\Gamma \subset \mathbb{P}(1,1,1,2)$ be defined by the three $2\times 2$ minors of the matrix above.\ The curve $\Gamma$ has generically two irreducible and reduced components, which can be easily checked on \verb|Magma|: one is rational, and the other has genus 4.\ We have that $\Phi'|_{Y_2}$ is a divisorial contraction to $\Gamma \subset \mathbb{P}(1,1,1,2)$.\ 
Since $\Gamma$ has two irreducible components, $\rho_X =2$.\

\paragraph{\#40671.\ } This case is completely analogous to the previous one.\ We give the $2\times 3$ matrix whose $2\times2$ minors define $\Gamma \subset \mathbb{P}(1,1,1,2)$.\ This is 
\begin{equation*}
\setlength\arraycolsep{7pt}
\begin{pmatrix}
y_1 & -y_2 & -y_3y_4\\
-y_2 y_3 & -y_1^2+y_2y_3 & -y_1y_3^3+y_2^4+y_4^2
\end{pmatrix} \; .
\end{equation*}

The curve $\Gamma$ has one rational irreducible component, and another irreducible component with genus 5.\ As before, $\rho_X =2$.\

\subsubsection{Wrapping up: conclusion of non-solidity proof} \label{sub:wrap}

We finalise the proof of non-solidity for the families whose link terminates with a divisorial contraction.\  

We prove that all models $X'$ in Tables \ref{tab:11111} and \ref{tab:1121} admit a structure of a strict Mori fibre space, with the possible exception of two.\

\begin{Lem}
Let $X$ be a family in Table \ref{tab:1121}.\ Then $X$ is non-solid.\
\end{Lem}
\begin{proof}
We consider the projection $\pi \colon X' \dashrightarrow \mathbb{P}^1$ in each case.\ Then, the generic fibre $S \subset X'$ is a surface for which $-K_S \sim \mathcal{O}_S(1)$ by adjunction since $X'$ has Fano index 2.\ The conclusion follows from Corollary \ref{cor:mainold}.\ 
\end{proof}

We show next that the families \#39890 and \#39928  in Table \ref{tab:11111} admit singular birational models in a family whose general member is a quasi-smooth complete intersection of a cubic and a quartic  $X'_{3,4} \subset \mathbb{P}(1,1,1,1,2,2)$.\ It was shown by Corti and Mella \cite{cortimella} that its general member is bi-rigid.\ In this case, we were not able to show that these are non-solid: these are the non-birationally rigid families mentioned in Theorem \ref{thm:mainC4I2version2}.\ 

\begin{Lem} \label{lem:solid}
Let $X$ be a quasi-smooth member of either family \#39890 or \#39928 in Table \ref{tab:11111}.\ Then $X$ is birational to a non-quasismooth complete intersection of a cubic and a quartic $X'_{3,4} \subset \mathbb{P}(1,1,1,1,2,2)$.\
\end{Lem}

\begin{Ex}
Consider the singular complete intersection 
\begin{equation*}
\begin{cases}
-t \xi  z_2 + t  z_2^2 + t  z_1 + \xi u +  z_1 y + u y =0\\
    t^3 \xi - \xi^4 + t y^3 - t^3  z_2 - t \xi u - \xi  z_1 y - t u y + t u  z_2 -  z_1^2 =0
\end{cases}
\end{equation*}
inside $\mathbb{P}(1,1,1,1,2,2)$ with homogeneous variables $t,\,\xi,\,y,\,z_2,\,z_1,\,u$.\ Then, there is a weighted blowup from the point $\mathbf{p_{z_2}}$ that initiates a birational link to a quasismooth family in Tom format with ID \#39928 (see Table \ref{tab:11111}).\ 
\end{Ex}

\begin{Lem} \label{lem:remaining}
Let $X$ be a quasi-smooth member of a family in Table \ref{tab:11111} not treated in Lemma \ref{lem:solid}.\ Then, the model $X'$ is birational to a del Pezzo fibration.\
\end{Lem}
\begin{proof}
Notice that each model $X'$ in Table \ref{tab:11111} has Fano index 1.\ General (quasi-smooth) members in each of these families of $X'$ have been treated in \cite{2raybrown}, \cite{okadaI} and \cite{CampoSarkisov}.\ In particular it was shown that these are non-solid.\ The same results extend to terminal families provided that the key monomials are still present in the equations of $X'$, which is the case for each model.\
\end{proof}

\bibliography{aaBibliograph}
\bibliographystyle{plain}

\end{document}